\documentclass[reqno, 11pt]{amsart}

\usepackage{amsfonts, amsthm, amsmath, amssymb}
\usepackage{hyperref}
\hypersetup{colorlinks=false}

\usepackage[margin=0.97 in]{geometry}

\usepackage{helvet}

\RequirePackage{mathrsfs} \let\mathcal\mathscr
  
\usepackage{mathtools}
\mathtoolsset{showonlyrefs}

\usepackage{hyperref}
\usepackage{enumerate}
\usepackage{color}

\numberwithin{equation}{section}

\newtheorem{theorem}{Theorem}[section] 
\newtheorem{lemma}[theorem]{Lemma}
\newtheorem{proposition}[theorem]{Proposition}
\newtheorem{corollary}[theorem]{Corollary}

\theoremstyle{definition}
\newtheorem*{acknowledgements}{Acknowledgements}

\newtheorem{definition}[theorem]{Definition}

\newtheorem*{notation}{Notation}

\newcommand{\spplit}[1]{\mathrm{Spl}\left(G_{#1}\right)}

\newcommand{\norm}{N}
\newcommand\qr[2]{\left(\frac{#1}{#2}\right)}

\newcommand{\abs}[1]{\left|#1\right|}

\newcommand{\absnorm}{\mathfrak{N}}

\newcommand{\e}{\mathrm{e}}

\renewcommand{\phi}{\varphi}

\newcommand{\ZZ}{\mathbb{Z}}

\newcommand{\NN}{\mathbb{N}}
\newcommand{\QQ}{\mathbb{Q}}

\newcommand{\CC}{\mathbb{C}}

\newcommand{\where}{\ :\ }

\renewcommand{\leq}{\leqslant}

\renewcommand{\geq}{\geqslant}

\DeclareMathOperator\ident{id}

\renewcommand{\c}{\mathbf{c}}

\renewcommand{\b}{\mathbf{b}}

\renewcommand{\r}{\mathbf{r}}

\DeclareMathOperator{\Gal}{Gal}

\DeclareMathOperator{\n}{N}
\DeclareMathOperator{\p}{P}
\DeclareMathOperator{\m}{M}

\DeclareMathOperator{\moo}{mod} 
\renewcommand{\bmod}[1]{\,(\moo{#1})}

\DeclareSymbolFont{bbold}{U}{bbold}{m}{n}
\DeclareSymbolFontAlphabet{\mathbbold}{bbold}

\newcommand{\md}[1]{  \left(\textnormal{mod}\ #1\right)}

\newcommand{\Q}{\mathbb{Q}}
\newcommand{\F}{\mathbb{F}}
\newcommand{\N}{\mathbb{N}}
\newcommand{\R}{\mathbb{R}}

\newcommand{\Z}{\mathbb{Z}}
\renewcommand{\l}{\left}
\renewcommand{\r}{\right}
\renewcommand{\b}{\mathbf}
\renewcommand{\c}{\mathcal}
\renewcommand{\epsilon}{\varepsilon}

\renewcommand{\leq}{\leqslant}
\renewcommand{\geq}{\geqslant}

\newcommand{\Da}{\mathfrak{D}_{\b{a}}}
\newcommand{\Daaa}{\mathfrak{D}_{(a,a,a)}}

\DeclareMathOperator{\HRH}{HRH}
\DeclareMathOperator{\HRHl}{HRH'}
\DeclareMathOperator{\consta}{\mathcal{A}_{\b{a}}}

\date{\today}

\begin{document}

\begin{abstract}
The first purpose of our paper is to show how Hooley's
celebrated method 
leading to his conditional proof of the Artin conjecture  
on primitive roots  
can be combined with the  
Hardy--Littlewood
circle method.
We do so by studying the 
number of representations of an odd integer as a sum of three
primes, all of which have prescribed primitive roots.
The second purpose is to analyse the singular series.
In particular,
using results of Lenstra, Stevenhagen and Moree,
we provide a partial factorisation as an Euler 
product and prove that this does not extend to a complete factorisation.
\end{abstract}

 \title
[
Vinogradov's theorem with primes
having given primitive roots
]
{
Vinogradov's three primes theorem with primes
 having given primitive roots
}

\author{C. Frei}
\address{University of Manchester\\   School of Mathematics\\  Oxford Road, Manchester M13 9PL, United Kingdom \\} \email{christopher.frei@manchester.ac.uk}

\author{P. Koymans}
\address{
Universiteit Leiden\\
Mathematisch Instituut\\
Niels Bohrweg 1\\
Leiden\\
2333 CA\\
Netherlands
}
\email{p.h.koymans@math.leidenuniv.nl}

\author{E. Sofos}
\address{
Max-Planck-Institut f\"{u}r Mathematik\\
Vivatsgasse 7\\
Bonn\\
53072\\
Germany
}
\email{sofos@mpim-bonn.mpg.de}

\maketitle
\vspace{-15 pt}

\setcounter{tocdepth}{1}
\tableofcontents

\vspace{-15 pt}

\section{Introduction}
\label{s:intro}
Can we represent an odd integer as a sum of $3$
odd 
primes all of which have $27$ as a primitive root?
Lenstra~\cite{MR0480413}
was the first to 
address the problem of
primes with 
a fixed primitive root
and lying in an 
arithmetic progression.
One of his results~\cite[Th.(8.3)]{MR0480413}
states that
if $b\neq 5 \md{12}$
then either there are no primes 
$p\equiv b \md{12}$ 
having $27$ as a
primitive root
or there is exactly one such 
prime, 
namely
$p=2$. 
Hence, unless $n\equiv 3 \md{12}$,
no such representation exists.

In this paper, we are interested in the converse direction, at least for all
sufficiently large values of $n$. The existence of infinitely many primes with a given primitive root $a$ is currently not known unconditionally
 for any $a\in\ZZ$, so we need to be content
with working under the assumption of a certain generalised Riemann Hypothesis, sometimes called \emph{Hooley's Riemann Hypothesis}. For any 
non-zero 
integer $a$, we will write $\HRH(a)$ for the hypothesis that
\begin{quote}
  for all square-free $k\in\NN$, the Dedekind zeta function of the number field $\QQ(\zeta_k, \sqrt[k]{a})$, where $\zeta_k\in\CC$ is a primitive $k$-th root of unity, satisfies the Riemann hypothesis.
\end{quote}

Our main theorem can be seen as a combination of the classical conditional result of Hardy and Littlewood \cite{MR1555183} towards ternary Goldbach with Hooley's \cite{hooley} conditional proof of Artin's conjecture. 
\begin{theorem}\label{thm:main}
  Let $\b{a}=(a_1,a_2,a_3)\in \Z^3$ 
such that no $a_i$ is $-1$ or a perfect square. Assuming $\HRH(a_i)$ for $i=1,2,3$, we have 
\begin{equation}
\label{eq:main}
\sum_{\substack{p_1+p_2+p_3=n
\\
\forall i:\ \F_{p_i}^*=\langle a_i \rangle}}
\prod_{i=1}^3
 \log p_i 
= \consta(n)n^2 + o(n^2),\quad \text{ as }\quad
n \to+\infty,
\end{equation} 
with an explicit factor 
$\consta(n)\in\R_{\geq 0}$
that satisfies $\consta(n)\gg_\b{a} 1$ whenever $\consta(n)>0$. 
\end{theorem}

The bulk of this paper will be devoted to the description and investigation of the factor $\consta(n)$. In particular, a product decomposition of $\consta(n)$ will allow us to interpret Theorem \ref{thm:main} as a local-global principle and gives the following as a simple consequence.

\begin{corollary}\label{cor:twentyseven}
  Assume $\HRH(27)$. Let $n$ be a sufficiently large odd integer. Then there are odd primes  
$p_1,p_2,p_3$ with $27$ as a primitive root and $n=p_1+p_2+p_3$ if and only if $n\equiv 3\bmod{12}$. 
\end{corollary}

We can also get an explicit saving in the error term, for the price of working under a stronger generalised Riemann hypothesis. Let $\HRHl(a)$ be the hypothesis that
\begin{quote}
 for each square-free $k>0$ all Hecke $L$-functions of the number field $\Q(\zeta_k, \sqrt[k]{a})$ satisfy the Riemann hypothesis.
\end{quote}
\begin{theorem}\label{thm:main_2}
    Let $a_1,a_2,a_3$ be three integers none of which is $-1$ or a perfect square. Assuming $\HRHl(a_i)$ for $i=1,2,3$, we have for $\beta\in(0,1)$,
\begin{equation}
\label{eq:main}
  \sum_{\substack{p_1+p_2+p_3=n
\\
\forall i:\ \F_{p_i}^*=\langle a_i \rangle}}
\prod_{i=1}^3
 \log p_i 
=\consta(n)
n^2
+O_{\b{a},\beta}(n^2 (\log n)^{-\beta})  
,\end{equation}
where the implied constant is effective 
and 
depends at most on $a_1,a_2,a_3$
and $\beta$.
\end{theorem}

Before returning to the explicit description of our factor $\consta(n)$, let us briefly review the relevant literature on Artin's conjecture and the ternary Goldbach problem, and introduce some necessary notation along the way.

\subsection{Artin's conjecture}
\label{s:artin}
Fix an integer $a\neq -1$ which is not a perfect
square.
A question going back to Gauss regards the
infinitude
of primes
having $a$ as a primitive root.
It was realised by Artin that the question admits an interpretation through algebraic number theory.
Denote by
$\zeta_k$ 
a primitive
$k$th root of unity
and 
define
for any positive 
square-free integer $k$
the number field
\begin{equation}
\label{def:gal01}
G_{a,k}:=
\Q(a^{1/k},\zeta_k)
.\end{equation}
Artin's criterion
states that  
the prime $p$ has  
$a$ as a primitive root 
if and only if 
for every prime $q$ 
the rational prime $p$ 
does not split completely in $G_{a,q}$. 
This led to the formulation of the following conjecture
via a collective effort due to Artin, Lehmer and Heilbronn.
Define 
\begin{align}
\label{def:del}
\Delta_a &:= \text{Disc}(\QQ(\sqrt{a}))\text{, the discriminant of $\QQ(\sqrt{a})$}\\
\label{def:hdel}
h_a&:=\max\big\{m \in \N:a \text{ is an } m\text{th power}\big\},\\
\label{def:cchh}
\c{A}_a&:=
\prod_{p\mid h_a} \l(1-\frac{1}{p-1}\r)
\prod_{p\nmid h_a} \l(1-\frac{1}{p(p-1)}\r)
\end{align}
and for positive integers $q$  let  
\begin{equation}
\label{def:dagsuo}
f_a^\ddagger(q):=
\Big(\prod_{\substack{p\mid q,p \mid h_a}}(p-2)^{-1}\Big)
\Big(\prod_{\substack{p\mid q,p\nmid h_a}}(p^2-p-1)^{-1}\Big)
.\end{equation} 
Here, and throughout our paper, 
the letter $p$ 
is reserved for rational primes.
We furthermore define 
\begin{equation}
\label{def:epicafe}
\c{L}_a:=
\c{A}_a
\cdot
\big(
1+
\mu(2|\Delta_a|)
f_a^\ddagger(|\Delta_a|)
\big)
,\end{equation} 
where $\mu$ is the M\"{o}bius function.
Artin's conjecture then states that  
\begin{equation}
\label{eq:falsecorrect}
\lim_{x\to+\infty}
\frac{\#
\big
\{p\leq x:
\F_{p}^*=\langle a \rangle
\big\}}
{\#\{p\leq x\}}
=
\c{L}_a
.\end{equation}
This conjecture 
is of 
substantial difficulty:
there is no 
value of $a$
for which the limit is
known to be positive.
In fact, it is not even known whether for every integer 
$a$
that is not a square or $-1$ there exists a prime having primitive root $a$.

A significant
step in the subject
has been the, conditional under $\HRH(a)$, 
resolution of Artin's conjecture
by Hooley~\cite{hooley}. 
His method is
pivotal 
in
the present work.
Notable
progress was later made by
Heath-Brown~\cite{MR830627},
who building on work of Gupta and Murty~\cite{MR762358},
showed 
unconditionally
that at least $\gg x/(\log x)^2$
primes $p\leq x$ 
have primitive root $q,r$ or $s$,
where $\{q,r,s\}$
is any set of non-zero integers which is multiplicative independent and such that 
none of $q,r,s,-3qr,-3qs,-3rs$ or $qrs$ is a square.
There is a rather extensive
list of
further results,
as well as 
certain cryptographic
applications;
the reader is referred to the comprehensive survey of Moree~\cite{moreesurvey}.
Lenstra~\cite{MR0480413} used Hooley's method
to show,
conditionally on $\HRH(a)$,
the existence of 
the
Dirichlet 
density of primes in an arithmetic progression
and with $a$ as primitive root. 
An explicit formula for these densities was given later by Moree~\cite{moreeprog}.
To describe Moree's result we need the following notation.
Let
\begin{equation}
\label{def:betain}
\beta_a(q)
:=
\begin{cases}  
(-1)^{\frac{\frac{\Delta_a}{\gcd(q,\Delta_a)}-1}{2}}\gcd(q,\Delta_a),
&\mbox{if } \frac{\Delta_a}{\gcd(q,\Delta_a)} \equiv 1 \md{2}\\ 
1 & \mbox{otherwise, }
\end{cases}
\end{equation}
and observe that $\beta_a(q)$ is a fundamental discriminant in case $\Delta_a/\gcd(q,\Delta_a)\equiv 1\bmod 2$. For positive integers $q$  let  
\begin{equation}
\label{def:dagsolo}
f_a^\dagger(q):=
\prod_{\substack{p\mid h_a,p \mid q }} \l(1-\frac{1}{p-1}\r)^{-1}
\prod_{\substack{p\nmid h_a,p \mid q }} \l(1-\frac{1}{p(p-1)}\r)^{-1}
.\end{equation} 
\begin{definition} 
\label{def:deltaa}
Assume that $a\neq -1$ is a non-square integer,
let $\Delta_a,h_a$ be as in~\eqref{def:del},~\eqref{def:hdel}
and assume that $x,q$ are integers with $q>0$. We define
\begin{equation}
\label{def:prolabainoume}
\c{A}_a(x\hspace{-0,3cm}\mod{q}):=
\c{A}_a \cdot
\begin{cases}  
\frac{f_a^\dagger(q)}{\phi(q)} \prod_{p|x-1, p| q}
\Big(1-\frac{1}{p}\Big),&\mbox{if } \gcd(x-1,q,h_a)=\gcd(x,q)=1,\\  
0,&\mbox{otherwise, } \end{cases}
\end{equation}
and
\[ 
\delta_a(x\hspace{-0,3cm}\mod{q}) 
:=
\c{A}_a(x\hspace{-0,3cm}\mod{q})
\Bigg(
1+\mu\l(\frac{2|\Delta_a|}{\gcd(q,\Delta_a)}\r)\left(\frac{\beta_a(q)}{x}\right)f_a^\ddagger\l(\frac{|\Delta_a|}{\gcd(q,\Delta_a)}\r) 
\Bigg)
,\]
where
$\phi(\cdot)$ is the Euler totient function
and
$\left(\frac{\cdot}{\cdot}\right)$ is
the Kronecker quadratic symbol.
\end{definition}
Moree's result~\cite{moreeprog}
states that, conditionally under $\HRH(a)$,
the Dirichlet
density of primes in an arithmetic progression
and with $a$ as primitive root
equals $\delta_a(x\hspace{-0,2cm}\mod{q})$.
His work will prove of central importance
in our interpretation of the 
Artin factor 
for the ternary 
Diophantine problem under study.

\subsection{Ternary Goldbach problem}
\label{s:vinogradov}
The ternary Goldbach problem has been one of the most central
subjects in   analytic number theory;
it asserts that every odd integer greater than $5$ is the sum of $3$ primes.
Hardy and Littlewood~\cite{MR1555183}
used the circle method
to provide the first 
serious approach to the problem;
they proved an asymptotic formula for 
the number of representations of $n$ as a sum of $k$ primes ($k\geq 3$)
conditionally on the
veracity of the generalised Riemann hypothesis.
This hypothesis was
removed later by 
Vinogradov~\cite{zbMATH03026053}.
His result states that  for every $\beta>0$ one has 
for all 
odd
integers $n$ that 
\[
\sum_{p_1+p_2+p_3=n}
\prod_{i=1}^3
\log p_i 
=
\frac{1}{2}
\Bigg(
\prod_{p}
\varrho_{p}(n)
\Bigg)
n^2
+O_\beta(n^2 (\log n)^{-\beta}) 
,\] 
where the
product is over all primes,
the
implied constant depends at most on $\beta$,
and
\begin{equation}
\label{eq:to2becomp}
\varrho_{p}(n)
:=
p
\Bigg(
\sum_{\substack{
b_1, b_2, b_3
\in 
(\Z/p\Z)^*
\\ b_1+b_2+b_3\equiv n \md{p}}} 
\frac{1}{(p-1)^3}  
\Bigg)
.\end{equation} 
This can be thought as the ratio of the probability that a random vector 
 $\b{b} \in ((\Z/p\Z)^*)^3$
satisfies $\sum_{1\leq i\leq 3} b_i\equiv n \md{p}$
to
the probability that a random vector 
 $\b{b} \in (\Z/p\Z)^3$
satisfies $\sum_{1\leq i\leq 3} b_i\equiv n \md{p}$, as made clear from
\begin{equation}
\label{eq:explanation?}
p=
\Bigg(\sum_{\substack{b_1,b_2,b_3 \md{p}\\ b_1+b_2+b_3\equiv n \md{p}}} \hspace{-0,2cm} \frac{1}{p^3} \Bigg)^{-1}
.\end{equation} 
It should be mentioned that 
Helfgott~\cite{helfgott}
recently
settled the ternary Goldbach problem.
Using recent developments in
additive combinatorics,
Shao~\cite{MR3165421}
provided general
conditions for an infinite subset $\c{P}$
of the primes 
that allow solving $n=p_1+p_2+p_3$ 
for large odd $n$ with each $p_i$ in $\c{P}$.
The 
result most related to our work 
is~\cite[Th.1.3]{MR3165421};
it states that if there exists $\delta>0$
such that 
the intersection of $\c{P}$
with each invertible
residue class modulo every integer $q$
has density at least 
$\delta/ \phi(q)$,
then,
under suitable
additional assumptions,
$n=p_1+p_2+p_3$ 
is soluble within~$\c{P}$.
This does not cover our situation, since
if $h_a>1$
then 
the densities 
$\delta_a(1\hspace{-0,2cm}\mod{h_a})$ vanish.
Furthermore, if $h_a=1$ then 
these densities
could become arbitrarily close to zero.
Indeed, if $q$ is of the form
$\prod_{p\leq T} p$
for some $T>2$  
then it is easy to see that 
\[ 
\delta_a(1\hspace{-0,3cm}\mod{q}) 
\phi(q)
\leq  
\prod_{p\leq T}
\Big(1-\frac{1}{p}\Big)
\ll
\frac{1}{\log \log q}
.\]
It would be  
interesting
to modify  
his 
approach in order to recover some of our results, 
for example a lower bound of the correct order of magnitude 
as the one provided by 
Theorem~\ref{thm:main}.
This approach would still require $\HRH(a_i)$
and besides the 
focal point of 
our paper 
is the
`Artin factor'
 $\consta(n)$
in
Theorem~\ref{thm:main}.
A further result related
to ours is that of
Kane~\cite{MR3060874}. 
A very 
special case of his work
provides an asymptotic 
for
the number of solutions of
$n=p_1+p_2+p_3$ 
when each $p_i$ 
lies in a prefixed Chebotarev class of a Galois extension of $\Q$. 
Primes with a prescribed
primitive root 
do admit a Chebotarev
description,
however the number of conditions involved
is not fixed.

\subsection{The factor $\consta(n)$}\ 
\label{s:results}

Let us now describe the representation of $\consta(n)$ that is obtained directly from the proof of Theorem \ref{thm:main}. Define for $q > 0$ and square-free $k > 0$  the number field $F_{a,q, k} := \mathbb{Q}(\zeta_q, \zeta_k, a^{1/k})$
, so that $G_{a,k}=F_{a,k,k}$. 
Moreover, for $b\in\ZZ$ with $\gcd(b,q)=1$, we let  $c_{a,q, k}(b):=1$ if the restriction of the automorphism $\sigma_b : \zeta_q \mapsto \zeta_q^b$ of $\mathbb{Q}(\zeta_q)$ to $\mathbb{Q}(\zeta_q) \cap G_{a,k}$ is the identity and we otherwise let $c_{a,q, k}(b):=0$. We use the usual notation $\e_q(z):= \exp(2 \pi i z/q)$, for $z\in \mathbb{C}, q\in \N$. The exponential sum
\begin{equation}
\label{def:exp99}
S_{a,q, k}(z) := \sum_{b \in (\Z/q\Z)^*} 
c_{a,q, k}(b) \e_q(zb)
\end{equation}
and the entities 
\begin{equation}
\label{def:exp9bs9}
L_{\b{a},q, \mathbf{k}}(z) := \prod_{i=1}^3
S_{a_i,q, k_i}(z),
\end{equation}
\begin{equation}
\label{def:podarakia}
d_{\b{a},\mathbf{k}}(q) :=
\prod_{i=1}^3
[F_{a_i,q, k_i} : \mathbb{Q}]
\end{equation}
will play a
central role throughout this paper. For positive square-free 
$k_1,k_2,k_3$ we define 
\begin{equation}\label{def:sing09}
\mathfrak{S}_{\b{a},\b{k}}(n):=
\sum_{q = 1}^\infty \frac{1}{d_{\b{a},\mathbf{k}}(q)} 
\sum_{\substack{z\in \Z/q\Z
\\
\gcd(z,q)=1}} \e_q(-n z) L_{\b{a},q, \mathbf{k}}(z)  
.
\end{equation} 
It will be made clear 
in~\S\ref{s:circle}
that this
is the \textit{singular series} for the representation
problem $n=p_1+p_2+p_3$ where for each $i$ the prime $p_i$ splits completely in $G_{a_i,k_i}$. The absolute convergence of the sum over $q$ will be verified in Lemma~\ref{lem:completion}. 
With this notation in place, the leading factor in Theorem \ref{thm:main} and Theorem \ref{thm:main_2} is given by  
\begin{equation}\label{eq:const_1}
  \consta(n)=\frac{1}{2}
\bigg(
\sum_{\b{k}\in \N^3}
\mu(k_1)
\mu(k_2)
\mu(k_3)
\mathfrak{S}_{\b{a},\b{k}}(n)
\bigg).
\end{equation}
The sum over $\b{k}$ will be shown to be absolutely convergent in Lemma~\ref{lem:completion}. It is
desirable 
to describe the integers $n$ for which $\consta(n)\neq 0$. 
An important remark is that
if the method of Hooley works in an Artin conjecture-related problem
then it provides a leading constant which is an infinite 
alternating sum of Euler products
that is not obviously equal to the conjectured Artin constant.
Such a phenomenon 
is well documented and 
can
be observed
for instance
in the work of Lenstra~\cite{MR0480413},
who studied the density of
primes in arithmetic progressions and
with a prescribed primitive root,
as well as the work of
Serre~\cite{sereos},
who studied the density of primes $p$
for which the reduction of an elliptic curve over
$\F_p$
is cyclic.
Artin constants have not been studied
in the context of 
Diophantine problems prior to the present
work,
however, we will show that
$\consta(n)$
factorises 
partially and we shall
provide an interpretation for $\consta(n)$.
For every positive integer $d$ we define  
the densities 
\begin{equation}
\label{def:sigmapenotdelta} 
\sigma_{\b{a},n}(d)
:=d
\Bigg(
\
\sum_{\substack{b_1,b_2,b_3 \md{d}\\ b_1+b_2+b_3\equiv n \md{d}}} 
\prod_{i=1}^3 
\frac{\delta_{a_i}(b_i\hspace{-0,3cm}\mod{d})}{\c{L}_{a_i}}
\
\Bigg)
.\end{equation}
The factor $d$ has an explanation that is identical to the explanation of the factor $p$ in~\eqref{eq:to2becomp}-\eqref{eq:explanation?}.
Let $[\cdot]$ denote the least common multiple,
$\nu_p(\cdot)$ be the $p$-adic valuation
and define 
\begin{equation}
\label{def:mytonga} 
\Da
:=
2^{
\min\{\nu_2(\Delta_{a_i}):1\leq i \leq 3\}
-
\max\{\nu_2(\Delta_{a_i}):1\leq i \leq 3\}
}
[\Delta_{a_1},\Delta_{a_2},\Delta_{a_3}].
\end{equation}

\vspace{10 pt}

\begin{theorem}
\label{thm:const}
The factor $\consta(n)$ in Theorems \ref{thm:main} and \ref{thm:main_2} factorises as follows,
\begin{equation}
\label{ee:grhtyuiop} 
\consta(n)=
\frac{1}{2}
\Big(\prod_{i=1}^3\c{L}_{a_i}\Big)
\sigma_{\b{a},n}(\Da)
\prod_{p\nmid \Da 
}
\sigma_{\b{a},n}(p)
.\end{equation}
Furthermore, whenever $\consta(n)>0$, we have 
 \begin{equation}
\label{th:grhreplaceall} 
\consta(n)\gg\prod_{i=1}^3\frac{\phi(h_{a_i})}{|
\Delta_{a_i}
|^2 h_{a_i}}
,\end{equation} 
with an absolute implied constant.
\end{theorem}
\vspace{10 pt}
For an interpretation of the right side of~\eqref{ee:grhtyuiop}
see~\S\ref{s:interdinner}. 
The proof of~\eqref{ee:grhtyuiop} 
(that will be provided in~\S\ref{proofofee:grhtyuiop})
requires adroit manoeuvring.
This
is
because 
the densities $\delta_a(b_i\hspace{-0,2cm}\mod{d})$ in~\eqref{def:sigmapenotdelta} 
have a complicated
dependence on $b_i$ and also do not 
exhibit
good factorisation properties 
with respect to $d$.

Let us 
furthermore
comment that
in contrast to the usual applications of the circle method,
the constant in~\eqref{ee:grhtyuiop}
does not factorise  
 as an 
Euler product,
see~\S\ref{s:nonfa}  
for a precise statement of this phenomenon. 
The following consequence of Theorem \ref{thm:main} and Theorem \ref{thm:const} can be interpreted as a local-global principle.
\vspace{10 pt}
\begin{corollary}\label{cor:loc_glob}
 Let $a_1,a_2,a_3$ be three integers none of which is $-1$ or a perfect square, and assume $\HRH(a_i)$ for $i=1,2,3$. For every sufficiently large odd integer $n$, the following statements are equivalent:
\begin{enumerate}
\item There are primes $p_1,p_2,p_3$ 
 not dividing $6\Delta_{a_1}\Delta_{a_2}\Delta_{a_3}$
such that 
each $a_i$ is
a primitive root modulo $p_i$ and $p_1+p_2+p_3 = n$.
\item For $d \in \{3, \Da\}$, 
there are primes $p_1,p_2,p_3$ with $\gcd(p_1p_2p_3,2d)=1$
such that $a_i$ is a primitive root for $p_i$ for every $i=1,2,3$ 
and $p_1+p_2+p_3 \equiv n \bmod d$. 
\end{enumerate}
\end{corollary}
\vspace{10 pt}
Though part $\emph{(2)}$ of Corollary \ref{cor:loc_glob} may not look like a purely local statement, it is one. In fact, for any $d$ in $\NN$, solubility of the congruence modulo $d$ in primes not dividing $2d$ 
with prescribed primitive roots
is equivalent to the  
statement that $\sigma_{\b{a},d}(n)>0$.  In Lemma \ref{lem:nmiddeducomeon2}, 
we shall see that $\sigma_{\b{a},n}(p)>0$ 
whenever $p\nmid  3 \Delta_{a_1}\Delta_{a_2}\Delta_{a_3}$. Moreover, it is clear from the definition in \eqref{def:sigmapenotdelta}, that whether $\sigma_{\b{a},d}(n)=0$ or not is a local condition modulo $d$.

\vspace{10 pt}

\subsection{Interpretation of the Artin factor for the ternary Goldbach problem}
\label{s:interdinner}
Studying the 
constants
 in any
counting problem
of flavour similar to that of Artin's conjecture
is a non-trivial task
and has been 
analysed
rather extensively.
The problems involve 
primes with a fixed primitive root,
primes in progressions and with a fixed primitive root
and primes such that the reduction of a fixed elliptic curve over the 
corresponding finite field is cyclic,
see the work of
Serre~\cite{sereos}.
The reader that is interested in an overview of the 
work that has been done 
on these constants so far is directed at the work of 
and 
Lenstra--Stevenhagen--Moree~\cite{entanglem},
as well as
the survey of Moree~\cite{moreesurvey}.

We now
focus on
the interpretation of the 
``Artin-factor"
$\c{A}_{\b{a}}(n)$
with the help of~\eqref{ee:grhtyuiop}.
First,
the factor $1/2$
is related to the density of solutions in $\R$
of $\sum_{1\leq i \leq 3}x_i=n$
and it has the exact same interpretation 
as in the classical situation of ternary Goldbach, and therefore,
we do not further
comment on this.

The term 
\[
\c{L}_{a_1}
\c{L}_{a_2}
\c{L}_{a_3}
\]
in~\eqref{ee:grhtyuiop}
should be thought of
as the ``probability"
that for all $i=1,2,3$,
a random prime $p_i$
has primitive root $a_i$, see~\eqref{eq:falsecorrect}.

The factors 
$\sigma_{\b{a},n}(d)$
for $d\in \{\Da\} \cup \{p \text{ prime }:p\nmid \Da\}$
admit an explanation that is comparable
to the analogous densities in the classical case of the ternary Goldbach
problem, 
see~\eqref{eq:to2becomp}.
There
is only one difference,
namely that 
  one has to use
the weight \[
\frac{\delta_{a_i}(b_i\hspace{-0,3cm}\mod{
d
})}{\c{L}_{a_i}} 
\]
instead of 
$1/(p-1)$.
This new 
weight 
equals the \textit{conditional}
probability that
a random 
prime lies in the arithmetic
progression
$b_i\md{d}$
given that it has primitive root $a_i$.

It would be desirable 
to use algebraic considerations 
(for example, the approach of 
 `entanglement' of splitting fields
  as in the work of 
Lenstra--Stevenhagen--Moree~\cite{entanglem}),  
to 
provide a prediction for $\c{A}_{\b{a}}(n)$
with a method that is different to the one in~\S\ref{proofofee:grhtyuiop}.

\vspace{10 pt}

\subsection{The case where all primitive roots are equal}
\label{s:sameroots}
In our next theorem, we provide an explicit description of 
the local conditions in Corollary \ref{cor:loc_glob},
but for space considerations we do so only in the important
case where 
\[
a_1=a_2=a_3=:a
.\]
The first row of 
the following table contains the discriminant of 
 $\Q(\sqrt{a})$
and the second row 
contains the power properties of $a$.
For example, if $a$ is a cube but not a fifth power we shall write 
$a\in \Z^3\setminus \Z^5$.

\vspace{10 pt}

\begin{theorem}
\label{propositionjoe}
Let $a\neq -1$ be a non-square integer and $n\in\NN$. Then the 'Artin factor'
$$\c{A}_{(a,a,a)}(n)$$
is strictly positive 
if and only if $n$ satisfies one of the congruence conditions 
in the third row of the following table.
\newline
\newline
\begin{tabular}{l*{6}{c}r}
$\mathrm{Disc}(\Q(\sqrt{a}))$  &  
Power properties of $a$     &  
Congruence conditions for $n$ \\
\hline
$-3$ & $\Z\setminus (\{-1\}\cup \Z^2)$   & $3\md{6}$  \\  
$-4$ & $\Z\setminus (\{-1\}\cup\Z^2)$  & $1\md{4}$   \\ 
$5$ & $\Z\setminus (\{-1\}\cup\Z^2)$   & $1\md{2}$  and not $0\md{5}$  \\  
$12$ & $\Z\setminus (\{-1\}\cup\Z^2\cup \Z^3)$   & $3,5,7,9\md{12}$  \\  
$12$ & $\Z^{3}\setminus (\{-1\}\cup \Z^2)$   & $3\md{12}$  \\ 
$-15$ & $\Z\setminus (\{-1\}\cup\Z^2\cup \Z^3 \cup \Z^5)$   & $1\md{2}$  and not $0\md{15}$  \\  
$-15$ & $\Z^3\setminus (\{-1\}\cup\Z^2\cup\Z^5)$   & $1\md{2}$  and $3, 6, 9, 12\md{15}$  \\  
$-15$ & $\Z^5\setminus (\{-1\}\cup\Z^2\cup \Z^3)$   & $1\md{2}$  and not $0,1,2,7,8,14\md{15}$  \\  
$-15$ & $\Z^{15}\setminus (\{-1\}\cup\Z^2)$   & $12\md{15}$  \\  
$-20$ & $\Z^5\setminus (\{-1\}\cup\Z^2)$   & $1\md{2}$  and not $1\md{20}$  \\  
$21$ & $\Z^7\setminus (\{-1\}\cup\Z^2\cup \Z^3)$   & $1\md{2}$  and not $8\md{21}$  \\   
$21$ & $\Z^3\setminus (\{-1\}\cup\Z^2\cup\Z^7)$   & $3\md{6}$   \\
$21$ & $\Z^{21}\setminus (\{-1\}\cup\Z^2)$   & $1\md{2}$  and $3,6,12,15\md{21}$  \\ 
$\pm 24$ & $\Z^{3}\setminus (\{-1\}\cup \Z^2)$   & $3\md{6}$  \\  
$60$ & $\Z^{3}\setminus (\{-1\}\cup \Z^2)$   & $3\md{6}$ \\ 
$60$ & $\Z^5\setminus (\{-1\}\cup\Z^2\cup\Z^3)$   & $1\md{2}$  and not $31,41\md{60}$  \\ 
$-84$ & $\Z^{3}\setminus (\{-1\}\cup \Z^2 )$   & $3\md{6}$  \\ 
$105$ & $\Z^{3}\setminus (\{-1\}\cup\Z^2)$   & $3\md{6}$  \\ 
$\pm 120$ & $\Z^{3}\setminus (\{-1\}\cup \Z^2)$   & $3\md{6}$  \\ 
$\pm 168$ & $\Z^{3}\setminus (\{-1\}\cup \Z^2)$   & $3\md{6}$  \\   
$-420$ & $\Z^{3}\setminus (\{-1\}\cup \Z^2)$   & $3\md{6}$  \\  
$\pm 840$ & $\Z^{3}\setminus (\{-1\}\cup \Z^2)$   & $3\md{6}$  \\    
other values $\setminus 3\Z$
& $\Z^{3}\setminus (\{-1\}\cup \Z^2)$   & $3\md{6}$  \\ 
every other value & $\Z\setminus (\{-1\}\cup \Z^2)$   & $1\md{2}$ 
\end{tabular}
\newline
\newline
The second to last row refers to all 
 integers
 $a$ not considered in a row above it, 
as long as $\mathrm{Disc}(\Q(\sqrt{a}))$ is not divisible by $3$.
The last row refers to every 
 integer
$a$ not considered in a row above it.
\end{theorem}
Theorem~\ref{propositionjoe} enables one
to
describe
 all large enough integers
having a representation as a sum of $3$ primes with a prescribed primitive root.
One such example is 
Corollary \ref{cor:twentyseven},
whose proof we give now.  
\begin{proof}[\textbf{Proof of 
Corollary \ref{cor:twentyseven}
}\!\!\!\!
]  
If $n$ is  a sum of $3$ odd
primes all of which have primitive root $27$,
we saw in
the first paragraph of our paper that $n$ must be $3\bmod{12}$. 
For the opposite direction we observe that if  
$a=27$ then
we have
$\mathrm{Disc}(\Q(\sqrt{a}))=12$
and $a \in \Z^{3}\setminus (\{-1\}\cup \Z^2)$,
hence alluding to the fifth row
in the table of 
Theorem~\ref{propositionjoe}
we see
that, conditionally 
on HRH$(27)$,
every sufficiently large integer $n\equiv 3 \md{12}$
is a sum of three 
odd
primes with primitive root $27$. 
\end{proof}

\subsection{Structure of the paper}
We study a generalisation of
the ternary Goldbach problem in~\S\ref{s:circle},
where
each of the three primes involved satisfies certain 
splitting conditions in a different
 number field extension of $\Q$.
The main result of~\S\ref{s:circle} is 
Proposition~\ref{prop:grhvino},
whose proof is given in~\S\ref{s:prpr}.

Next,~\S\ref{s:opening}
contains the first steps for the 
combination of Hooley's argument~\cite{hooley} and the Hardy--Littlewood
circle method.
Theorem~\ref{thm:main}
will be proved 
in~\S\ref{proofofee:ggrrhh00},
while 
Theorem~\ref{thm:main_2}
is verified
in~\S\ref{proofofee:ggrrhh11}.

The rest of our 
paper, namely~\S\ref{s:local2glob},
deals with the
`Artin factor'
$\consta(n)$.
The former 
part of 
Theorem~\ref{thm:main},
viz.~\eqref{ee:grhtyuiop},
is verified in~\S\ref{proofofee:grhtyuiop},
while the latter 
part,
viz.~\eqref{th:grhreplaceall},
is established in~\S\ref{proofofth:grhreplaceall}.
 Corollary~\ref{cor:loc_glob}
and
Theorem~\ref{propositionjoe}
are proved in~\S\ref{proofcoro}
and \S\ref{s:simplifying}
respectively.
Finally, we show that  $\consta(n)$
does not factorise as an Euler product
in~\S\ref{s:nonfa}.

\begin{notation}
 The letters 
$p$
and $\ell$ will
always denote a rational 
prime.
The entities
$a_i,h_{a_i},\Delta_{a_i}$ are considered constant 
throughout our work,
thus
the dependence of 
implied constants 
on them
will not be recorded.
On several occasions 
our implied constants are absolute,
this will always be specified.  
Finally, we will use the notation
\[
\e(z) := \exp(2 \pi i z)
\text{ and }  
\e_q(z) := \exp(2 \pi i z/q),
(
z\in \mathbb{C},
q\in \N)
.\]
\end{notation} 
  
\begin{acknowledgements}
This work 
was 
completed while Christopher Frei and Peter Koymans
were visiting the Max Planck Institute in Bonn,
the hospitality of which
is greatly acknowledged.
\end{acknowledgements}

\section{Uniform ternary Goldbach with certain splitting conditions}
\label{s:circle}

In this section the letters $k,k_i$ shall refer exclusively to 
positive square-free
integers.
Recall~\eqref{def:gal01}
and define 
\begin{equation}
\label{def:spl}
\spplit{a,k}
:=\{p  \text{ prime in } \NN
: 
p   \text{ splits completely in }  G_{a,k}\}
.\end{equation}  
We study the asymptotics of the representation function
\begin{equation}
\label{def:vv44}
V_{\b{a},\b{k}}(n)
:=
\sum_{\substack{p_1+p_2+p_3=n\\
\forall i:\ p_i \in\spplit{a_i,k_i}
}
}
\prod_{i=1}^3 \log p_i
.\end{equation}
We will see that the singular series related to the estimation of $V_{\b{a},\b{k}}(n)$ is the series $\mathfrak{S}_{\b{a},\b{k}}(n)$ introduced in \eqref{def:sing09}. Kane~\cite{MR3060874} studied a very general set of problems, one case of which is that 
of evaluating $V_{\b{a},\b{k}}(n)$
asymptotically. 
His work
provides
a function
$f_\b{a}$ 
such that for each 
$B>0$ and 
square-free
$k_1,k_2,k_3$  
we have 
\begin{equation}
\label{eq:kaneinef}
V_{\b{a},\b{k}}(n)
=
\frac{1}{2}
\mathfrak{S}_{\b{a},\b{k}}(n)
n^2
+O_B
\Bigg(
|f_{\b{a}}(\b{k})|
\frac{n^2}{
(\log n)^B}
\Bigg)
,\end{equation}
where the implied constant depends at most on $\b{a}$ and $B$. 
This can be deduced by taking $$N:=n,\ 
X:=n,\  k:=3,\  a_i:=1,\  K_i:=G_{a_i,k_i}\ 
\text{ and }\ 
C_i:=\mathrm{id}_{G_{a_i,k_i}}$$ 
in~\cite[Th.2]{MR3060874}.
With this choice the constant $C_\infty$  in~\cite[Th.2]{MR3060874}
equals $n^2/2$ and a long but straightforward computation allows one to show that 
the `singular series' $\mathfrak{S}_{\b{a},\b{k}}(n)$
can be factored
into the remaining parts
of the main term in the asymptotic formula~\cite[Eq.(1.2)]{MR3060874}.

Our aim in this section is to prove the following result, conditional on the hypothesis $\HRHl(a_i)$ introduced before Theorem \ref{thm:main_2}. It constitutes
a version of~\eqref{eq:kaneinef}
that has a power saving in the error term and an explicit and polynomial dependence on the $k_i$.
As is surely familiar to circle method
experts,
an error term of this quality 
is currently out of reach 
unconditionally even in the setting 
of the
classical
ternary Goldbach problem.
\begin{proposition}
\label{prop:grhvino}
Assume $\HRHl(a_i)$ for $i=1,2,3$.
The following estimate holds for all
square-free
$k_1,k_2,k_3$
with $1\leq k_1,k_2,k_3 \leq n $  
and with an implied constant 
depending at most on $\b{a}$,
\[
V_{\b{a},\b{k}}(n)
=
\frac{1}{2}
\mathfrak{S}_{\b{a},\b{k}}(n)
n^2
+
O\Big(
 n^{11/6}  
(\log n)^6 
\big(\max_{1\leq i \leq 3} k_i\big)^{6} 
\Big)
.\] 
\end{proposition}

\subsection{Algebraic considerations}
\label{s:apm}
We shall need explicit bounds for certain algebraic quantities
associated to $G_{a,k}$.  This subsection is mostly devoted to
providing the necessary estimates.

Recall the definitions of $\Delta_a$ and $h_a$,
given in~\eqref{def:del} and~\eqref{def:hdel}.
We begin by determining 
the degree of the number field $F_{a,q, k}$ defined at the start of \S\ref{s:results} 
(see \cite[Lemma 2.3]{moreeprog}).
\begin{lemma}
\label{lem:bythegive}
For $k$  square-free 
let $k'= k/\gcd(k, h_a)$. Then 
$
[F_{a,q, k} : \QQ] = 
k'
\phi([q, k])/
\epsilon(q, k)
$,
where 
\begin{equation*}
\epsilon(q, k)=
  \begin{cases}
    2, &\text{ if } 2 \mid k \ \text{ and } \ \Delta_a\mid [q, k],\\
    1, &\text{ otherwise.}
  \end{cases}
\end{equation*}
\end{lemma} 

\begin{lemma}\label{lDisc}
Let $k' = k/\gcd(k, h_a)$ and $a = g_1^{\gcd(k,h_a)} g_2^{k}$, with $g_1$ free of $k'$-th powers. Then 
\begin{equation*}
\frac{
\log\abs{\mathrm{Disc}(F_{a,q, k})}
}{
[F_{a,q, k}:\QQ]
} 
\leq \log k' + \log([q, k]) + 2\log \abs{g_1}
.
\end{equation*}
\end{lemma}
 
\begin{proof}
We have $\abs{\mathrm{Disc}(F_{a,q, k})} = \absnorm(\Delta_{F_{a,q, k}/\QQ(\zeta_{[q, k]})})
 |\mathrm{Disc}(\QQ(\zeta_{[q, k]}))|^{[F_{a,q, k} : \QQ(\zeta_{[q, k]})]}$,
where $\absnorm$ is the absolute norm of an ideal and $\Delta_{F_{a,q, k}/\QQ(\zeta_{[q, k]})}$ is the relative discriminant ideal. Any $k'$-th root  $\alpha\in F_{a,q,k}$ of $g_1$ generates $F_{a,q,k}$ over $\QQ(\zeta_{[q,k]})$, so it's different $d(\alpha)\neq 0$ is in the different ideal of $F_{a,q, k} / \QQ(\zeta_{[q, k]})$. Since the minimal polynomial of $\alpha$ over $\QQ(\zeta_{[q,k]})$ divides $x^{k'}-g_1$, we find that $k'\alpha^{k'-1}$ is a multiple of $d(\alpha)$ in $\mathcal{O}_{F_{a,q,k}}$, and thus in the different ideal as well. Hence, 
\begin{equation*}
\absnorm(\Delta_{F_{a,q, k}/\QQ(\zeta_{[q, k]})}) \leq |\norm_{F_{a,q,k}/\QQ}(k'\alpha^{k' - 1})| \leq (k')^{[F_{a,q,k}:\QQ]} |g_1|^{(k' - 1) \varphi([q, k])}\leq (k')^{[F_{a,q,k}:\QQ]} |g_1|^{2[F_{a,q,k}:\QQ]}.
\end{equation*}
To complete the proof, use $|\mathrm{Disc}(\QQ(\zeta_{[q, k]}))| = [q, k]^{\varphi([q, k])}\prod_{p \mid qk} p^{-\varphi([q, k])/(p-1)}\leq [q, k]^{\varphi([q, k])}$.
\end{proof}

Clearly, the intersection $\QQ(\zeta_q) \cap G_{a,k}$ contains $\QQ(\zeta_{\gcd(q, k)})$.
More precisely, it is determined as follows (see \cite[Lemma 2.4]{moreeprog}). 
 
\begin{lemma}
\label{lem:intersection}
We have
\begin{equation*}
[\QQ(\zeta_q) \cap G_{a,k} : \QQ(\zeta_{\gcd(q, k)})] =
      \begin{cases}
        2 &\text{ if } 2\mid k, \ \Delta_a \nmid k \text{ and } \Delta_a \mid [q, k]\\
        1 &\text{ otherwise.}
      \end{cases}
\end{equation*}
In the first case,  
the integer $\beta_a(q)$ defined in \eqref{def:betain} is a fundamental discriminant and we have
$\QQ(\zeta_q) \cap G_{a,k} = \QQ(\zeta_{\gcd(q, k)},\sqrt{\beta_a(q)})$. 
\end{lemma}
 
  Since both
$\QQ(\zeta_q)$ and $G_{a,k}$ are normal, the same holds for their
compositum $F_{a,q, k}$. We investigate the existence of certain
elements of the Galois group $\Gal(F_{a,q, k}/\QQ)$. Recall the definitions of $\sigma_b$ and $c_{a,q,k}(b)$ from the start of \S\ref{s:results}.
 
\begin{lemma}
\label{lem:extend_auto}
Let $b\in\ZZ$ with $\gcd(b,q)=1$. The following are equivalent:
\begin{enumerate}
\item there is an automorphism $\sigma \in \Gal(F_{a,q, k}/\QQ)$ with 
\begin{equation}
\label{eq:sigma00restr11cond}
\sigma|_{\QQ(\zeta_q)} = \sigma_b \ \text{ and } \ \sigma |_{G_{a,k}}=\ident_{G_{a,k}},  
\end{equation}
\item $c_{a,q,k}(b) = 1$,
\item with $\beta_a(q)$ defined in \eqref{def:betain}, we have
\begin{align}
&b \equiv 1 \md{\gcd(q, k)}, \quad \text{ and } \label{eq:condtyty1}\\
&2 \mid k, \ \Delta_a \nmid k, \ \Delta_a \mid [q, k] \ \text{ implies that }\ \qr{\beta_a(q)}{b}= 1. \label{eq:condqwerty2}
\end{align}
\end{enumerate}
Moreover, if $\sigma$ as in (1) exists, it is unique and in the center of $\Gal(F_{a,q, k})/\QQ$.
\end{lemma}

\begin{proof}
Write $I := \QQ(\zeta_q) \cap G_{a,k}$. The map $\sigma \mapsto (\sigma|_{\QQ(\zeta_q)},\sigma|_{G_{a,k}})$ provides an isomorphism  
\begin{equation*}
\Gal(F_{a,q, k}/\QQ) 
\cong 
\{(\sigma_1,\sigma_2) \in \Gal(\QQ(\zeta_q)/\QQ) \times \Gal(G_{a,k}/\QQ)\where \sigma_1|_I=\sigma_2|_I\}.
\end{equation*}
Thus, an automorphism $\sigma$ with \eqref{eq:sigma00restr11cond} exists if and only if $c_{a,q,k}(b)=1$, proving the equivalence of \emph{(1)} and \emph{(2)}.
In this case $\sigma$ is necessarily unique and clearly in the center of $\Gal(F_{a,q, k}/\QQ)$, because the Galois group $\Gal(\QQ(\zeta_q)/\QQ)$ is abelian and $\ident_{G_{a,k}}$ is in the center of $\Gal(G_{a,k}/\QQ)$. Thus, let us study the conditions under which $c_{a,q,k}(b)=1$.

Since $\QQ(\zeta_{\gcd(q, k)}) \subset I$ and $\sigma_b|_{\QQ(\zeta_{\gcd(q, k)})}$ coincides with the automorphism
$\zeta \mapsto \zeta^{b \md{\gcd(q, k)}}$, the condition \eqref{eq:condtyty1} is clearly necessary. Thus, we assume it to hold from now on, whence $\sigma_b|_{\QQ(\zeta_{\gcd(q, k)})} = 
\ident_{G_{a,k}}$. 
If the antecedent in \eqref{eq:condqwerty2} is false, then we have $I = \QQ(\zeta_{\gcd(q, k)})$ by Lemma \ref{lem:intersection}, and thus $c_{a,q,k}(b)=1$. If the antecedent in \eqref{eq:condqwerty2} holds, then, invoking Lemma \ref{lem:intersection} once more, we find that $\sqrt{\beta_a(q)}\in \QQ(\zeta_q)$ and $c_{a,q,k}(b)=1$ is equivalent to
\begin{equation}\label{eq:restrQquadratic}
\sigma_b(\sqrt{\beta_a(q)}) = \sqrt{\beta_a(q)}.
\end{equation}
Since $\beta_a(q)$ is a fundamental discriminant, we may invoke \cite[Lemma 2.2]{moreeprog} to see that \eqref{eq:restrQquadratic} is equivalent to $\qr{\beta_a(q)}{b} = 1$.
\end{proof}
 
\subsection{Consequences of $\HRHl(a)$}
\label{s:drawing}
In this section we 
use the hypothesis $\HRHl(a)$
to provide estimates for certain 
exponential sums
related to the estimation of
$V_{\b{a},\b{k}}(n)$.

\begin{lemma}
\label{lebkov666}
Assume $\HRHl(a)$. For any square-free $k$
and coprime integers $c,q$ 
we have
\[
\sum_{\substack{p \leq x \\ p \in 
\spplit{a,k} 
}} (\log p) \e_q(cp) = 
\frac{x}{\varphi(q) [G_{a,k}: \mathbb{Q}]} 
\sum_{\substack{\chi \md q \\ \chi \circ \absnorm = \chi_0}} \overline{\chi(c)} \tau(\chi) + O(k^2
 \sqrt{qx} 
(\log qx)^2).
\]
Here, $\chi$ runs through all Dirichlet characters modulo $q$ for which $\chi\circ\absnorm$, considered as a ray class character modulo $q\mathcal{O}_{G_{a,k}}$, is the trivial ray class character $\chi_0$. Moreover, $\tau(\chi)$ denotes the Gauss sum
$\tau(\chi) = \sum_{y \md{q}} \chi(y) \e_q(y)$.
\end{lemma}

\begin{proof}
We have
\begin{equation}
\label{eFA}
\sum_{\substack{p \leq x \\ p \in \spplit{a,k}}} 
(\log p) \e_q(cp) = 
\sum_{\substack{p \leq x,  p\nmid q \\ p \in 
\spplit{a,k}
}} (\log p)\e_q(cp) + O((\log q)^2).
\end{equation}
Bringing into play the Dirichlet characters modulo $q$
allows us to inject, for $p\nmid q$,
\[
\e_q(cp)
= \frac{1}{\varphi(q)} \sum_{b \md{q}} \sum_{\chi \md{q}} \chi(b) \overline{\chi(cp)} 
\e_q(b)
= \frac{1}{\varphi(q)} \sum_{\chi \md{q}} \overline{\chi(cp)} \tau(\chi)
\]
into~\eqref{eFA},
thus acquiring the validity of 
\begin{equation}
\label{eNF}
\sum_{\substack{p \leq x \\ p \in 
\spplit{a,k} }} (\log p)\e_q(cp)
= \frac{1}{\varphi(q)} \sum_{\chi \md{q}} \overline{\chi(c)} \tau(\chi) 
\psi_{a,k}(x, \overline{\chi}) + O((\log q)^2),
\end{equation}
where  
\begin{align*}
\psi_{a,k}(x, \chi) 
&:= \sum_{\substack{p \leq x \\ p \in \spplit{a,k}  
}} (\log p) \chi(p) 
= \frac{1}{[G_{a,k} : \mathbb{Q}]} \sum_{\substack{\absnorm\mathfrak{p} \leq x \\ \deg(\mathfrak{p}) = 1}} 
(\log\absnorm\mathfrak{p}) \chi(\absnorm\mathfrak{p})\\
&= \frac{1}{[G_{a,k} : \mathbb{Q}]} \sum_{\absnorm\mathfrak{n} \leq x} 
\Lambda(\mathfrak{n}) \chi(\absnorm\mathfrak{n}) + O(\sqrt{x} \log x).
\end{align*}
Here and for the rest of this section $\mathfrak{p}$ denotes a prime ideal in $\mathcal{O}_{G_{a,k}}$, $\deg(\mathfrak{p})$ denotes its inertia degree over $\QQ$, $\mathfrak{n}$ denotes an ideal in $\mathcal{O}_{G_{a,k}}$, and $\Lambda$ is the von Mangoldt function on ideals of $\mathcal{O}_{G_{a,k}}$, defined by $\Lambda(\mathfrak{p}^e):=\log\absnorm\mathfrak{p}$ for $e\geq 1$ and $\Lambda(\mathfrak{n}):=0$ in all other cases. Observing that $\chi\circ\absnorm$ defines a character of the ray class group of $G_{a,k}$ modulo $q\mathcal{O}_{G_{a,k}}$, we consider its Hecke $L$-function,
\[
L(s, \chi) := \sum_{\mathfrak{n} \neq 0} \chi(\absnorm\mathfrak{n}) (\absnorm\mathfrak{n})^{-s}
.\]
It is now easy to see that 
$
-L'(s, \chi)/L(s, \chi) = \sum_{\mathfrak{n} \neq 0} \Lambda(\mathfrak{n}) \chi(\absnorm\mathfrak{n}) (\absnorm\mathfrak{n})^{-s}
$.
The Ramanujan--Petersson conjecture is obviously true for $L(s, \chi)$, since it is true for any Hecke $L$-function. Hence Theorem 
$5.15$ from \cite{MR2061214}
implies that
\[
\sum_{\absnorm\mathfrak{n} \leq x} \Lambda(\mathfrak{n}) \chi(\absnorm\mathfrak{n}) = r_\chi x + O(x^{\frac{1}{2}} (\log x) \log(x^{[G_{a,k} : \mathbb{Q}]} \mathfrak{q}(\chi))),
\]
where $r_\chi$ is the order of the pole of $L(s, \chi)$ at $s = 1$. For the definition of $\mathfrak{q}(\chi)$, see page $95$ of \cite{MR2061214}. 
Furthermore, on page 129 of \cite{MR2061214} it is proven that
$\mathfrak{q}(\chi) \leq 4^{[G_{a,k} : \mathbb{Q}]} |\mathrm{Disc}(G_{a,k})| q^{[G_{a,k} : \mathbb{Q}]}$.
Our next task is to make explicit the value of $r_\chi$. 
If $\chi \circ \absnorm$ is the trivial ray class character $\chi_0$ modulo $\mathcal{O}_{G_{a,k}}$, then we have $r_\chi = 1$; otherwise we have $r_\chi = 0$. Using $|\tau(\chi)| \leq \sqrt{q}$ and Lemma \ref{lDisc} we can substitute in (\ref{eNF}) to find that
\[
\frac{1}{\varphi(q)} \sum_{\chi \md{q}} \overline{\chi(c)} \tau(\chi) \psi_{a,k}(x, \overline{\chi}) = \frac{x\varphi(q)^{-1}
}{
[G_{a,k} : \mathbb{Q}]} \sum_{\substack{\chi \md{q} \\ \chi \circ \absnorm = \chi_0}} \overline{\chi(c)} \tau(\chi) + O([G_{a,k}
: \mathbb{Q}] \sqrt{qx} (\log qx)^2),
\]
thus concluding our proof upon observing that 
$
 [G_{a,k}:\Q]=
 [F_{a,k,k}:\Q]
\leq
k^2$.
\end{proof}

Although it is possible to directly evaluate the main term in Lemma \ref{lebkov666}, 
we will instead use the following trick.

 \begin{lemma}
\label{lTrick}
Under the same conditions as in Lemma~\ref{lebkov666} we have
\[
\sum_{\substack{p \leq x \\ p \in \spplit{a,k}}} 
(\log p) \e_q(cp) 
= \frac{S_{a,q, k}(c) }{[F_{a,q, k} : \mathbb{Q}]} x
+ o_{q, k}(x), \text{ as } x\to+\infty.  
\]
\end{lemma}

\begin{proof}
Partitioning 
in progressions modulo $q$
we see that, owing to~\eqref{eFA}, 
the 
sum over $p$ in our lemma
is equal to the following quantity up to an error of size
$o_{q, k}(x)$, 
\[
\sum_{b \in (\Z/q\Z)^*}
\e_q(bc)
\sum_{\substack{p \leq x \\ p \equiv b\md{q} \\ p \in \spplit{a,k}}} \log p.
\]
By Lemma \ref{lem:extend_auto} there exists an automorphism $\sigma$ of $F_{a,q, k}$ satisfying
\[
\sigma|_{\QQ(\zeta_q)} = \sigma_b \ \text{ and } \ \sigma|_{G_{a,k}} = \ident_{G_{a,k}}
\]
if and only if $c_{a,q, k}(b) = 1$. Furthermore, if such an automorphism exists, it is unique. The lemma is now a consequence of Chebotarev's density theorem.
\end{proof}

Combining Lemma \ref{lebkov666} and Lemma \ref{lTrick}  
proves the following 
lemma.

\begin{lemma}
\label{l1}
Under the same assumptions as in Lemma~\ref{lebkov666} we have
\[
\sum_{\substack{p \leq x \\ p \in \spplit{a,k}}} (\log p) 
\e_q(cp) 
= \frac{S_{a,q, k}(c) x}{[F_{a,q, k} : \mathbb{Q}]} + O(k^2 \sqrt{qx} \log ^2qx).
\]
\end{lemma} 

Define for a square-free integer 
$k > 0$ the exponential sum 
\begin{equation}
\label{def:thesum}
f_{a,k}(\alpha) = \sum_{\substack{p \leq n \\ p \in \spplit{a,k}
}} (\log p) \e( \alpha p ),
\ (\alpha \in \R).
\end{equation} 

The next lemma 
is easily
proved 
via partial summation 
and Lemma \ref{l1}.

\begin{lemma}
\label{lebkov22}
Assume $\HRHl(a)$. Let $k$ be square-free integer and define 
$\alpha = c/q + \beta$, where $(c, q) = 1$. Then
\[
f_{a,k}(\alpha) = \frac{S_{a,q, k}(c)}{[F_{a,q, k} : \mathbb{Q}]} \int_0^n \e(\beta x) 
\mathrm{d}x
 + O\left(k^2(1 + |\beta| n) \sqrt{q n} (\log qn)^2
\right).
\]
\end{lemma}
 
It will be necessary
to
gain a better understanding 
of the
exponential sums
$S_{a,q, k}(c)$.
We start by studying $c_{a,q,k}(\cdot)$
in the next lemma,
whose proof   flows directly from~\eqref{eq:condtyty1}
and~\eqref{eq:condqwerty2}. 
\begin{lemma}
\label{lcqka}
Let $b,q$ be coprime integers 
and factor $q$ as
$
q = d \prod_{i = 1}^l 
p_i^{e_i}
$
with $d$ an integer composed of primes dividing $\Delta_a$ and $p_i$ distinct prime numbers 
not dividing $\Delta_a$.
Then we have for any square-free integer $k$,
\[
c_{a,q, k}(b) = c_{a, d, k}(b) \prod_{i = 1}^l 
c_{a,{p}^{e_i}_i, k}(b).
\]
\end{lemma}

\begin{lemma}
\label{lSkqa1}
Let $k$ be  square-free,
assume that  $b, q $ are coprime  integers
and
suppose that $q = q_1 q_2$,
$b = b_1q_2 + b_2q_1$,
with $q_1,q_2$ coprime.
If $\gcd(q_1, \Delta_a) = 1$ or $\gcd(q_2, \Delta_a) = 1$
then we have 
\[
S_{a,q, k}(b) = S_{a,q_1, k}(b_1) S_{a,q_2, k}(b_2).
\]
\end{lemma}

\begin{proof}
By the Chinese remainder theorem
we can write
each element 
$y\in\Z/q\Z$ 
as 
$y_1q_2+y_2q_1$,
where
$y_i \in \Z/q_i\Z$,
thus 
showing
that 
$\e_q(by)=
\e_{q_1}(b_1y_1q_2)
\e_{q_2}(b_2y_2q_1)
$.
This leads to 
\[
S_{a,q, k}(b) 
=
\sum_{y \in (\Z/{q}\Z)^*} c_{a,q, k}(y) \e_q(by) 
= 
\sum_{y_1 \in (\Z/{q_1}\Z)^*}  
\e_{q_1}(b_1y_1q_2)
\sum_{y_2 \in (\Z/{q_2}\Z)^*}  
\e_{q_2}(b_2y_2q_1)
c_{a,q, k}(y_1q_2 + y_2q_1) 
.\] 
By
Lemma~\ref{lcqka}
we have  $c_{a,q, k}(y_1q_2 + y_2q_1) = 
c_{a,q_1, k}(y_1q_2 + y_2q_1) 
c_{a,q_2, k}(y_1q_2 + y_2q_1)
$.
The entity $c_{a,q,k}(y)$ is periodic $\md{q}$ as a function of $y$,
thus we can write 
$S_{a,q, k}(b)$
as
\[
\sum_{y_1 \in (\Z/{q_1}\Z)^*}  
\e_{q_1}(b_1y_1q_2)
c_{a,q_1, k}(y_1q_2) 
\sum_{y_2 \in (\Z/{q_2}\Z)^*}   
\e_{q_2}(b_2y_2q_1)
c_{a,q_2, k}(y_2q_1)
\]
and a simple linear change of variables
in each sum
completes the proof. 
\end{proof}

\begin{lemma}
\label{lSkqa2}
For
$k$ square-free,
$b$ an integer
and
$p$ a prime with $p\nmid b
\Delta_a$ 
we have
\[
|S_{a,p^j, k}(b)| =
\left\{
	\begin{array}{ll}
	       1, & j = 1 \\
		0, & j > 1.
	\end{array}
\right.
\]
\end{lemma}

\begin{proof}
Let us observe that \eqref{eq:condqwerty2} always holds for $q=p^j$ as in the lemma, as the antecedent is never satisfied. 
We first handle the case $j = 1$. 
If $p\nmid k$
then
by Lemma \ref{lem:extend_auto},
$S_{a,p, k}(b)$ is the classical
Ramanujan sum
and the result follows,
while in the remaining case,
$p\mid k$, 
the result is 
also 
immediate from~\eqref{eq:condtyty1}.
Now suppose $j > 1$. Again, 
if $p\nmid k$, the sum in the lemma
is
a Ramanujan sum and the result follows.
We are therefore free to assume that
$p\mid k$.
Writing 
$y=1+px$ we see that 
\[
S_{a,p^j, k}(b)=
\sum_{\substack{y\md{p^j} \\ y \equiv 1\md{p}}} 
\e_{p^j}(by) 
= 
\e_{p^j}(b) 
\sum_{x \md{p^{j - 1}}} 
\e_{p^{j-1}}(bx) 
,\]
which is clearly sufficient
since the inner sum vanishes.
\end{proof}

\begin{lemma}
\label{lem:exceptions}
Let $r,Q,c \in \Z$
be such that 
$
rQ\neq 0,
\gcd(c,Q)=1
$,
$r$ divides $Q$
and 
assume that a function
$f:\Z
\to \mathbb{C}$ 
has
period $|r|$.
If we have $|r|<|Q|$ then the following sum vanishes,
\[
\sum_{b\md{|Q|}} \e_{|Q|}(bc) f(b)
.\]
\end{lemma}
\begin{proof}
The claim becomes clear upon writing the sum in our lemma as
\[
\sum_{b_0\md{|r|}} \e_{|Q|}(b_0c) f(b_0)
\sum_{x\md{|Q/r|}} \e_{|Q/r|}(xc)
\]
and observing that if $|Q/r|\neq 1$
then 
each
exponential sum over $x$ vanishes.
\end{proof}

\begin{lemma}
\label{lSkqa3}
Let $k$ be a square-free integer,
suppose that $q$ is composed of primes dividing $\Delta_a$ and let $b$ be an integer with $\gcd(b, q) = 1$.
If $q\nmid \Delta_a$,
then $S_{a,q, k}(b) = 0$.

\end{lemma}
\begin{proof}
First suppose $2 \nmid k$ or $\Delta_a \mid k$ or $\Delta_a \nmid [q, k]$ and write
$q=p^{e_1}_1 \cdots p^{e_l}_l$. 
We have 
\[
c_{a,q, k}(b) = 
\prod_{i = 1}^l 
c_{a,p^{e_i}_i, k}(b)
,\]
therefore 
$S_{a,q, k}(b) = 0$ can now
be
easily
proved
as before, 
as our hypotheses imply that $e_j>1$ for at least one $j$.

Now suppose that $2 \mid k$ and $\Delta_a \nmid k$ and $\Delta_a \mid [q, k]$. For $y\in\ZZ$, let $f(y):=1$ if $y\equiv 1\bmod{\gcd(k,q)}$ and $\left(\frac{\beta_a(q)}{y}\right)=1$, and $f(y):=0$ otherwise.
By Lemma \ref{lem:extend_auto}
we have
\[
S_{a,q, k}(b) = 
\sum_{\substack{y\md{q}}}f(y)\e_q(by).
\]

Since $\gcd(k,q)\mid \gcd(\Delta_a,q)=|\beta_a(q)|$ and $\beta_a(q)$ is a fundamental
discriminant, we see that $f$ has period $\gcd(\Delta_a,q)$, strictly dividing $q$ by our hypotheses. Apply Lemma \ref{lem:exceptions}. 
\end{proof}

Combining 
Lemmas \ref{lSkqa1}, \ref{lSkqa2} and~\ref{lSkqa3} 
allows us to conclude
that  
\begin{equation}
\label{eq:unifbb}
S_{a,q, k}(b) \ll 1,
\end{equation}
where the implied constant depends at most on $a$.

\subsection{Proof of Proposition~\ref{prop:grhvino}}
\label{s:prpr}
Recall~\eqref{def:thesum}.
Our starting point is the circle method identity,

\begin{equation}
\label{eq:hardylittle}
\sum_{\substack{p_1 + p_2 + p_3 = n \\ p_i \in \text{Spl}(G_{a_i,k_i})}} \prod_{i=1}^3(\log p_i)
=
\int_0^1 f_{a_1,k_1}(\alpha) f_{a_2,k_2}(\alpha) f_{a_3,k_3}(\alpha) \e(-n\alpha) 
\mathrm{d}\alpha
.\end{equation}
  
 \begin{corollary}
\label{lebkov223}
Assume $\HRHl(a)$, and suppose $\alpha,c,q$ fulfil
$
|\alpha - c/q| \leq
q^{-1}n^{-2/3}$,
$\gcd(c, q) = 1$,
$q \leq n^{2/3}$
and
that
$k$ is square-free.
Then we have $
f_{a,k}(\alpha) \ll
(n/q+ k^2 n^{5/6})
( \log n)^{2}
$.
 
\end{corollary}
\begin{proof}
Observe that Lemma~\ref{lem:bythegive} gives
$
[F_{a,q, k} : \mathbb{Q}]^{-1}
\ll  \varphi([q,k])^{-1} \leq
 \varphi(q)^{-1}
\ll
(\log q)q^{-1}$,  
hence, 
by 
Lemma~\ref{lebkov22}
and~\eqref{eq:unifbb}
one obtains
$
f_{a,k}(\alpha) \ll
n (\log n)
q^{-1}
 +  k^2
(1 + n^{1/3}q^{-1})
\sqrt{q n} (\log n)^2
$.
Our proof can then be concluded by using
$q\leq n^{2/3}$.
\end{proof}
Define
$P:=n^{\nu}$,
for an absolute  constant $\nu \in (0,1/6]$  
that will be chosen
later. 
In our situation 
the major arc
$\mathfrak{M}(c,q)$
is defined
for coprime $c,q$ 
via
\[
\mathfrak{M}(q,c)
:=\{\alpha: |\alpha-c/q| \leq q^{-1} n^{-2/3}
\}
,\]
while we  let 
$\mathfrak{M}$
be the union of all $\mathfrak{M}(q,c)$
with 
$1\leq q \leq P$,
$1\leq c \leq q,\gcd(c,q)=1$
and define the minor arcs through
$\mathfrak{m}:=[0,1]\setminus \mathfrak{M}$.
We
note here
that the major arcs are disjoint owing to
$(q q')^{-1}>(qn^{2/3})^{-1}+(q'n^{2/3})^{-1}$ that can be proved for all $n>8$
due to
$q,q'\leq  n^{1/3}$. 
\begin{corollary}
\label{lebkov777} 
Assume $\HRHl(a_i)$ for $1\leq i\leq 3$. Then 
\begin{equation*}
\int_\mathfrak{m} 
|f_{a_1,k_1}(\alpha)
f_{a_2,k_2}(\alpha)
f_{a_3,k_3}(\alpha) 
| 
\mathrm{d}\alpha 
\ll
n^{2- \nu}
( \log n)^{3}
\min_i k_i^2 
.
\end{equation*}
\end{corollary}
\begin{proof}
By Dirichlet's approximation theorem,
for each $\alpha$
there exist coprime integers
$c,q$ with $|\alpha-c/q|\leq q^{-1}n^{-2/3}$
and $1\leq q\leq n^{2/3}$.
If $\alpha \in \mathfrak{m}$
then $q>n^\nu$, hence
Corollary~\ref{lebkov223} 
yields the estimate
$f_{a,k}(\alpha) \ll k^2 n^{1-\nu}(\log n)^{2}$. 
We may assume 
$k_1\leq k_2,k_3$
with no loss of generality,
therefore the integral in our lemma is 
$\ll 
k_1^2 n^{1-\nu} 
( \log n)^{2}
\int_0^1 |f_{a_2,k_2}(\alpha) f_{a_3,k_3}(\alpha)| \mathrm{d}\alpha 
$, thus 
Cauchy's inequality
yields the following bound for the last integral,
\[  \ll 
\left(\int_0^1 |f_{a_2,k_2}(\alpha)|^2 \mathrm{d}\alpha \r)^{1/2}
\left(\int_0^1 |f_{a_3,k_3}(\alpha)|^2 \mathrm{d}\alpha \r)^{1/2}
.\]
Both integrals are at most    $ 
\sum_{p\leq n } (\log p)^2
\ll n \log n
$,  
which provides the desired result. 
\end{proof}
Note that 
if $\beta+c/q \in \mathfrak{M}(q,c)$ for some $q\leq n^{1/3}$
then 
Lemma \ref{lebkov22}
shows that  
\[
f_{a_i,k_i}(\alpha) = 
\frac{S_{a_i,q, k_i}(c)}{[F_{a_i,q, k_i} : \mathbb{Q}]} 
\int_0^n \e(\beta x) 
\mathrm{d}x
 + 
O\left(
\frac{n^{5/6}}{q^{1/2}} 
(\log n)^2
\max_i k_i^2
\right)
.\]
Hence
the estimates 
\[
\int_0^n \e(\beta x) \mathrm{d}x 
\ll
\min\{n, |\beta|^{-1}\}
\quad\text{ and }\quad
 \frac{S_{a,q, k}(c)}{[F_{a,q, k} : \mathbb{Q}]}
\ll \varphi(q)^{-1}
\]
show that $f_{a_1,k_1} (c/q + \beta)
f_{a_2,k_2} (c/q + \beta)
 f_{a_3,k_3} (c/q + \beta) 
-L_{\b{a},q, \mathbf{k}}(c)
d_{\b{a},\mathbf{k}}(q)^{-1}
\left(\int_0^n \e(\beta x) \mathrm{d}x \right)^3 
$
is
\begin{equation}
\label{eq:maxell}
\ll
\frac{\min\{n^2, |\beta|^{-2}\}}{\varphi(q)^2}
\frac{n^{5/6}}{q^{1/2}} 
(\log n)^2
\max_i k_i^{2}
+\frac{n^{15/6}}{q^{3/2}} 
(\log n)^6
\max_i k_i^{6} 
.\end{equation} 
The major arcs make the following contribution
towards~\eqref{eq:hardylittle},
\[
\sum_{1\leq q \leq n^\nu} 
\sum_{\substack{1 \leq c \leq q  \\ \gcd(c,q)=1}} 
\int_{
-q^{-1}
n^{-2/3}
}^{
q^{-1}
n^{-2/3} 
} 
f_{a_1,k_1} (c/q + \beta)
f_{a_2,k_2} (c/q + \beta)
 f_{a_3,k_3} (c/q + \beta) 
\e(-n(c/q + \beta)) 
\mathrm{d}\beta
,\]
and a  
straightforward
analysis utilising~\eqref{eq:maxell}
reveals that
 the last expression equals  
\[
\sum_{1\leq q \leq n^\nu}  
\sum_{\substack{1 \leq c \leq q  \\ \gcd(c,q)=1}} 
\frac{\e_q(-cn) L_{\b{a},q, \mathbf{k}}(c) }{d_{\b{a},\mathbf{k}}(q)}
\int_{
-q^{-1}
n^{-2/3}
}^{
q^{-1}
n^{-2/3} 
} 
\left(\int_0^n \e(\beta x) \mathrm{d}x \right)^3 
\e(-n \beta) 
\mathrm{d}\beta
+
O\l(
\frac{ n^{11/6}  (\log n)^6
}
{
\max_i k_i^{-6}
}
\r)
.\]
The integral over $\beta$ 
can be estimated as 
$n^2/2
+O(q^2 n^{4/3})
$, thus by
~\eqref{eq:unifbb}
the sum over $q$ 
is 
$
\mathfrak{S}_{\b{a},\b{k}}(n)
n^2/2
+O((n^{4/3+\nu}+n^{2-\nu})(\log n)^3)
$ 
and the choice $\nu = 1/6$ 
  concludes the proof of 
Proposition~\ref{prop:grhvino}.

\section{Injecting the circle method into Hooley's approach}
\label{s:injecthardhoo}

\subsection{Opening phase}
\label{s:opening}

The aim of~\S\ref{s:injecthardhoo} is to prove Theorem~\ref{thm:main} and Theorem~\ref{thm:main_2}. 
We commence in this subsection
by calling upon
parts of Hooley's work~\cite{hooley}
that will prove
useful.
We will make an effort to keep the notation in line 
with his as much as possible. 
In this section, the letters $p,q$ will be reserved for primes. 
Two primes $p,q$ 
are said to satisfy the property $R_a(q,p)$
if both of the following conditions hold,
\[
q|(p-1);
a \text{ is a } q\text{th power residue } \md{p}
.\]
A standard index calculus argument 
shows that for a prime $p\nmid a$
the integer
$a$ is a 
primitive root $\md{p}$
if and only if $R_{a}(q,p)$ fails for all primes $q$.
For any 
$\eta,\eta_1,\eta_2 \in \R_{>0}$
we define  
\[
\n_a(n,\eta):=
\#
\big\{
p\leq n:
R_a(q,p)  \text{ fails for all primes }
q \leq \eta 
\big\}
\]
and 
\[
\m_a(n,\eta_1,\eta_2):=
\#
\big\{
p\leq n: \text{ there exists } q\in (\eta_1,\eta_2]
\text{ such that } R_a(q,p) \text{ holds} 
\big\}
.\]
Letting 
\[
\n_a(n):=\#\{p\leq n:a \text{ is a primitive root modulo } p \}
\]
we see from the work of Hooley~\cite[Eq.(1)]{hooley}
that for each $\xi_1,\xi_2,\xi_3 \in \R$ with 
\[
1\leq \xi_1<\xi_2<\xi_3<n-1
\]
we have 
\begin{equation}
\label{hoo:ena}
\n_a(n)=\n_a(n,\xi_1)
+O\big(
\m_a(n,\xi_1,\xi_2)
+
\m_a(n,\xi_2,\xi_3)
+
\m_a(n,\xi_3,n-1)
\big)
.\end{equation}
Hooley makes specific choices for the parameters $\xi_i$;
we will keep the same choice for $\xi_2$ and $\xi_3$,
namely
$
\xi_2:=n^{\frac{1}{2}} (\log n)^{-2}$,
$
\xi_3:=n^{\frac{1}{2}} \log n$,
however, we shall  later
choose a different value for $\xi_1$.
For the moment we shall only demand that
$
1<\xi_1 \leq 
(\log n) (\log \log n)^{-1} 
$.
The estimates proved in~\cite[Eq.(2), Eq.(3)]{hooley}
provide us with 
\begin{equation}
\label{hoo:dio}
\n_a(n)=\n_a(n,\xi_1)
+O\big(
\m_a(n,\xi_1,\xi_2)
+n (\log \log n) (\log n)^{-2}
\big)
.\end{equation}
The argument in~\cite[Eq.(33)]{hooley}
shows that for each $\xi_1$ as above,
one has 
under
$\HRH(a)$
that 
\[
\m_a(n,\xi_1,\xi_2)
\ll
\frac{n}{\log n} \sum_{q>\xi_1}\frac{1}{q^2}
+\frac{n}{\log ^2 n}
,\]
which, once combined with the
simple estimate $\sum_{q>\xi_1}q^{-2}\ll \xi_1^{-1}$
and~\eqref{hoo:dio}
provides us with
\begin{equation}
\label{hoo:tria}
\n_a(n)=\n_a(n,\xi_1)
+O\Bigg(
\frac{n  }{\log n}
\frac{1}{\xi_1} 
+ \frac{n \log \log n}{\log ^2n} 
\Bigg)
,\end{equation}
with an implied constant depending at most on $a$.
\begin{lemma}
\label{lem:exceptions34}
For any
$\beta \in (0,1)$
and any
sets of primes $\c{P}_i\subset [1,n]$  
of cardinality $\epsilon(\c{P}_i) n /\log n$ 
the following estimate holds with an implied constant that depends at most on $\beta$,
\[
\sum_{\substack{p_1+p_2+p_3=n\\
\exists i: p_i \in \c{P}_i 
}}
\prod_{i=1}^3
\log p_i
\ll_\beta
n^2
(\max_i \epsilon(\c{P}_i))^{\beta} 
.\]
\end{lemma}
\begin{proof}
Define $r_2(m):=\#\{(p_1,p_2): 
p_i \text{ prime},
p_1+p_2=m\}$.
The sum in the
lemma is at most 
\[
(\log n)^3 
\sum_{i=1}^3
\sum_{\substack{p_1+p_2+p_3=n\\
p_i\in \c{P}_i 
}}
1
=
(\log n)^3 
\sum_{i=1}^3
\sum_{p<n} 
\mathbf{1}_{\c{P}_i}(p)
r_2(n-p)
\]
and using
H\"{o}lder's inequality 
with exponents $(1/\beta,1/(1-\beta))$
allows us to bound
the inner sum on the right by
\begin{equation*}
\epsilon(\c{P}_i)^{\beta} n^{\beta} (\log n)^{-\beta}
(
\sum_{p < n}
r_2(n-p)^{1/(1-\beta)}
)^{1-\beta}
.
\end{equation*}
Straightforwardly, there exists $c=c(\beta)>0$
with
$(1-z)/(1-2z)\leq (1+cz)^{1-\beta}$
for all $0<z\leq 1/3$.
Using this 
for $z=1/p'$
and alluding to
the 
following
classical
bound
(that can be found in~\cite[Eq. (7.2)]{MR0424730}, for example),
\[ 
r_2(m)
\ll
\frac{m}{(\log m)^2}\prod_{\substack{p' | m, p' \neq 2}}\frac{p'-1}{p'-2}
\]
yields 
\[r_2(m)
\ll_\beta
\frac{m}{(\log m)^2}
\prod_{p' | m}\l(1+\frac{c}{p'}\r)
^{1-\beta}
.\]
Therefore
the
quantity in the lemma is  
\[
\ll (\log n)^3
\Big(
\frac{n\max_i \epsilon(\c{P}_i) }{\log n}
\Big)^{\beta} 
\Big(
\Big(\frac{n}{(\log n)^2}
\Big)^{1/(1-\beta)}
\sum_{p < n}
\prod_{p' | n-p} (1+c/p')
\Big)^{1-\beta}
\]
and 
to finish our proof
it remains to show
that 
\[
\sum_{p < n}
\prod_{p' | n-p} (1+c/p')
\ll_c
\frac{n}{\log n}
.\]
Rewriting this sum as 
$ 
\sum_{d\leq n} 
\mu(d)^2 c^{\omega(d)}
d^{-1}
\#\{p<n:p\equiv n \md{d}\}
$
we see that 
the contribution from integers $d>n^{1/2}$ is 
$
\ll
\sum_{n^{1/2}<d\leq n} 
c^{\omega(d)}d^{-1} 
(n/d+1) 
\ll
n^{1/2+1/100}$.
By Brun--Titchmarsh,
the contribution of
terms with  
$d\leq n^{1/2}$
is $
\ll
n (\log n)^{-1}
\sum_{d\leq n^{1/2}} 
c^{\omega(d)}
(d\phi(d))^{-1}
\ll
n (\log n)^{-1}
$,
thus concluding our proof.
\end{proof}    
Let us define the set 
\[
\c{P}_i 
:= 
\big\{p:p|a_i\big\}
 \cup 
\big\{
p\leq n: 
R_{a_i}(q,p) \text{ holds for some prime } 
q>\xi_1
\big\}
.\] 
The arguments bounding $\m_{a}(n,\xi_1,n-1)$ in the deduction of \eqref{hoo:tria} show under $\HRH(a)$ that
\begin{equation}
\label{hoo:tess}
\#\c{P}_i
\ll
\frac{n  }{\xi_1\log n}
+ \frac{n \log \log n}{\log ^2n} 
.\end{equation}
We can now apply Lemma~\ref{lem:exceptions34}
and to do so let us observe that by~\eqref{hoo:tess}
we have 
\[
\epsilon(\c{P}_i)
=\frac{\log n}{n}
\#\c{P}_i
\ll
\frac{1}{\xi_1}
+\frac{\log \log n}{\log n}
\ll
\frac{1}{\xi_1}
.\]
Therefore, under $\HRH(a_i)$
for $i=1,2,3$,
 and 
for each fixed 
$\beta \in (0,1)$ 
we
acquire the validity of   
\begin{equation}
\label{eq:bridge}
\sum_{\substack{p_1+p_2+p_3=n
\\
\forall i:\ \F_{p_i}^*=\langle a_i \rangle}}
\prod_{i=1}^3
 \log p_i 
= 
\hspace{-0,3cm}
\sum_{\substack{
p_1+p_2+p_3=n,
p_i \nmid a_i
\\
\forall i,
\forall q\leq \xi_1:\ 
R_{a_i}(q,p_i) \text{ fails}
}
}
\prod_{i=1}^3 \log p_i
+
O_\beta
\Big(
\frac{n^2}{\xi_1^{\beta}}
\Big)
.\end{equation}  
Bringing into play 
the following quantity 
for each square-free positive integer $k_i$,
\begin{equation}
\label{def:gal}
\p_{\b{a},\b{k}}(n)
:=
\sum_{\substack{
p_1+p_2+p_3=n,\ 
p_i \nmid a_i
\\
\forall i:\ q|k_i \Rightarrow R_{a_i}(q,p_i) \text{ holds} 
}
}
\prod_{i=1}^3 \log p_i,
\end{equation}
makes the following estimate available,
once the inclusion-exclusion principle 
has been used,
\begin{equation}
\label{eq:finn}
\sum_{\substack{p_1+p_2+p_3=n
\\
\forall i:\ \F_{p_i}^*=\langle a_i \rangle}}
\prod_{i=1}^3
 \log p_i 
=
\hspace{-0,4cm}
\sum_{\substack{\b{k} \in \N^3\\
p| k_1k_2k_3
\Rightarrow
p\leq \xi_1
}}
\hspace{-0,5cm}
\mu(k_1)
\mu(k_2)
\mu(k_3)
\p_{\b{a},\b{k}}(n)
+
O_\beta\Big(
n^2
\xi_1^{-\beta}
\Big)
.\end{equation}
The
entity 
$\p_{\b{a},\b{k}}(n)$
is analogous to 
$\mathrm{P}_{\!a}(k)$
that is present in the work of 
Hooley~\cite[\S 3]{hooley}.
Indeed,
the inclusion-exclusion
argument above
is inspired by the argument leading to~\cite[Eq.(5)]{hooley}.

Using the arguments in~\cite[\S 4]{hooley}
we shall first translate the $R_{a_i}(q,p_i)$-condition present in~\eqref{def:gal}
into a condition related to the factorisation properties of the prime $p_i$ 
in certain number fields.  
Recall the definition of $h_a$ given in~\eqref{def:hdel}.
For any positive square-free integer $k_i$
we  define 
$
k'_i:=k_i/\gcd(k_i,h_{a_i})
$.
Then, as explained in~\cite[Eq.(8)]{hooley},
for a prime $p\nmid a_i$ and a square-free integer $k_i$,
the conditions $R_{a_i}(q,p)$ hold for all $q \mid  k_i$ 
if and only if 
\[
x^{k'_i}\equiv a_i \md{p}
\text { is soluble\quad and\quad }
p\equiv 1 \md{k_i}
.\]
It is then proved following~\cite[Eq.(8)]{hooley}
that, in light of
the Kummer--Dedekind theorem,
this is in turn 
equivalent to the property that 
 $p$
is completely split in the number field 
$
\Q( 
a_i^{1/k'_i},\zeta_{k_i})
$.
Recall~\eqref{def:gal01}
and let us see
why 
\[
G_{a_i,k_i}=
\Q(a_i^{1/k_i'},\zeta_{k_i})
.\]
It is clearly sufficient to show that 
$a_i^{1/{k_i}} \in \Q(a_i^{1/k'_i},\zeta_{k_i}
)$.
Writing 
$a_i=b^{h_{a_i}}$
and using 
$\mu({k_i})^2=1$, we see that
$\gcd(h_{a_i}\gcd({k_i},h_{a_i}),{k_i})| 
h_{a_i}
$, hence there are integers $x,y$ 
with
$$h_{a_i}\gcd(k_i,h_{a_i})x+ {k_i} y= h_{a_i}.$$
This leads to the equality
$a_i^{1/{k_i}}=
(b^{1/{k_i}})^{h_{a_i}}=
b^y 
(a_i^{1/{k_i}'})^x
$, which completes the argument.

Recalling the definition of $\spplit{a_i,k_i}$ in \eqref{def:spl}, 
we infer by~\eqref{def:gal}
that for all $\b{k}\in \N^3$ with 
each $k_i$  square-free we have

\[
\p_{\b{a},\b{k}}(n)
=
\sum_{\substack{
p_1+p_2+p_3=n,\ 
p_i \nmid a_i
\\
\forall i:\ p_i \in \spplit{a_i,k_i} 
}
}
\prod_{i=1}^3 \log p_i =  V_{\b{a},\b{k}}(n) + O_\beta(n^2((\log n)/n)^\beta),\]
for any $\beta\in(0,1)$. For the second equality, recall \eqref{def:vv44} and use Lemma \ref{lem:exceptions34}. Injecting this into~\eqref{eq:finn}
we have proved that 
whenever 
 $1<\xi_1\leq (\log n) (\log \log n)^{-1}$ 
and $0<\beta<1$
then
\begin{equation}
\label{eq:finnope}
\sum_{\substack{p_1+p_2+p_3=n
\\
\forall i:\ \F_{p_i}^*=\langle a_i \rangle}}
\prod_{i=1}^3
 \log p_i 
=
\hspace{-0,4cm}
\sum_{\substack{\b{k} \in \N^3\\
p| k_1k_2k_3
\Rightarrow
p\leq \xi_1
}}
\hspace{-0,5cm}
\mu(k_1)
\mu(k_2)
\mu(k_3)
V_{\b{a},\b{k}}(n)
+
O_\beta\Big(
n^2
\xi_1^{-\beta}
\Big)
,\end{equation}
where, for $2-\beta<\delta<2$, the estimate 
\begin{align*}
\sum_{\substack{\b{k} \in \N^3\\
p| k_1k_2k_3
\Rightarrow
p\leq \xi_1
}}
\hspace{-0,5cm}
|\mu(k_1)
\mu(k_2)
\mu(k_3)|
n^\delta
&\leq
n^\delta
\Big(
\hspace{-0,2cm}
\sum_{\substack{k \in \N\\
p| k
\Rightarrow
p\leq \xi_1
}}
|\mu(k)|
\Big)^3
=
n^\delta 2^{3\#\{p\leq \xi_1\}}
\\&\leq 
n^\delta \e^{3\xi_1}
\leq
n^{\delta+\frac{3}{\log \log n}}
\\&\ll_{\beta,\delta}
n^2 (\log n)^{-\beta}  (\log \log n)^\beta 
\leq
n^2 \xi_1^{-\beta} 
\end{align*}      
Before concluding the proofs of Theorem~\ref{thm:main} and Theorem~\ref{thm:main_2},
we need a preparatory lemma.

\begin{lemma}
\label{lem:completion}
The series defining $\mathfrak{S}_{\b{a},\b{k}}(n)$ in \eqref{def:sing09} and representing $\consta(n)$ in \eqref{eq:const_1} are absolutely convergent. For each 
$\epsilon>0$
and
$z\geq 1$ we have 
\[
\sum_{\substack{\b{k}\in \N^3\\\exists i,p:\ p|k_i \text{and } p \geq z}}
\hspace{-0,5cm}
|\mathfrak{S}_{\b{a},\b{k}}(n)|
\Big(\prod_{i=1}^3|\mu(k_i)|\Big)
\leq 
\hspace{-0,3cm}
\sum_{\substack{\b{k}\in \N^3\\\exists i:\ k_i \geq z}}
\Big(\prod_{i=1}^3|\mu(k_i)|\Big)
\sum_{q = 1}^\infty \frac{1}{d_{\b{a},\mathbf{k}}(q)} 
\sum_{x\in (\Z/q\Z)^*}
\hspace{-0,3cm}
|L_{\b{a}, q, \mathbf{k}}(x)|  
\ll_\epsilon
\frac{1}{z^{1-\epsilon}}
,\]  
with an implied constant depending at most on  $\b{a}$ and $\epsilon$. 
\end{lemma}
\begin{proof}
The first inequality is clear
by \eqref{def:sing09}.
Observe that 
$k'_i\geq k_i/h_{a_i} \gg 
k_i$,
hence by Lemma~\ref{lem:bythegive}
we obtain 
\[\frac{1}{d_{\b{a},\b{k}}(q)} \ll 
\prod_{i=1}^3 \frac{1}{k_i \varphi([q,k_i])}
=
\frac{1}{\varphi(q)^3}
\prod_{i=1}^3 
\frac{\varphi(\gcd(q,k_i))}{k_i 
\varphi(k_i)} 
.\]
Combining this with~\eqref{eq:unifbb}
we see by~\eqref{def:sing09}
that for $\epsilon>0$
and square-free $k_i$,
\begin{align*}
\sum_{q = 1}^\infty \frac{1}{d_{\b{a},
\mathbf{k}}(q)} 
\sum_{x\in (\Z/q\Z)^*}
|L_{\b{a},q, \mathbf{k}}(x)| 
&\ll
\prod_{i=1}^3 \frac{1}{k_i \varphi(k_i)}
\sum_{q=1}^\infty
\frac{\varphi(\gcd(q,k_1))\varphi(\gcd(q,k_2))\varphi(\gcd(q,k_3))}{\varphi(q)^2}
\\
&\ll_\epsilon
\frac{ \gcd(k_1,k_2,k_3)}
{(k_1k_2k_3)^{2-\epsilon} }
.\end{align*}
Therefore, the inner sum our lemma is 
\[
\ll
\sum_{k_1 \geq z}
\hspace{-0,2cm} 
\frac{|\mu(k_1)|}{k_1^
{2-\epsilon}
}
\sum_{k_2\in \N}
 \frac{|\mu(k_2)|}{k_2^{2-\epsilon}}
\sum_{k_3\in \N}
\frac{|\mu(k_3)|\gcd(k_1,k_2,k_3)}{k_3^{2-\epsilon}}
.\]
Using the estimates
\[
\sum_{k_3\in \N} |\mu(k_3)| \gcd(k_3,m) {k_3}^{\!-2+\epsilon}\ll_\epsilon m^\epsilon\quad\text{ and }\quad
\
\
\sum_{k_1\geq  z}
\hspace{-0,2cm} 
\frac{|\mu(k_1)|}{k_1^
{2-\epsilon}}
\ll z^{-1+\epsilon}
\]
concludes our proof of the desired bound, which implies absolute convergence of the sum in $\eqref{eq:const_1}$.
\end{proof}

\subsection{The proof of Theorem~\ref{thm:main}}
\label{proofofee:ggrrhh00}

Recall~\eqref{eq:kaneinef}.
Now note that,
replacing $f_{\b{a}}(\b{x})$ by a larger function if necessary,   
we may assume in the statement of~\eqref{eq:kaneinef}
that 
$f_{\b{a}}([1,\infty)^3)$ is a 
subset of
$(1,\infty)$. 
Fix any $B>0$. The function
 
$$x \mapsto 
\log(1+x)
+
\sum_{1\leq k_1,k_2,k_3 \leq x}
\hspace{-0,3cm}
f_{\b{a}}(\b{k}),$$  
is strictly increasing, hence it has an inverse, say $h_{\b{a}}
(x)$.
Define the function $\xi_1:(1,\infty)\to \R$ through

\begin{equation}
\label{eq:obsobs0}
\xi_1(x):=
\frac{1}{2}\cdot\min 
\left\{\frac{\log x}{\log\log x},\ 
\log(h_{\b{a}}((\log x)^{B/2}))
\right\}
\end{equation}
and observe that 
\begin{equation}
\label{eq:obsobs1}
\lim_{x\to+\infty}
\xi_1(x)
=+\infty
,\end{equation}
however, owing to the non-explicit 
error term in~\cite[Th.2]{MR3060874}
we
cannot
have any
further
control on the rate of divergence in the last limit.
For $n\gg 1$, the definition of $\xi_1$ 
implies
\[ 
\sum_{1\leq k_1,k_2,k_3 
\leq \e^{2\xi_1(n)}}
f_\b{a}(\b{k})
\leq (\log n)^{B/2}
.\] 
Noting that a square-free integer
with all of its prime factors bounded by $\xi_1(n)$
must be
at most
$\prod_{p\leq \xi_1(n)}p\leq \exp(2 \xi_1(n))$
and injecting~\eqref{eq:kaneinef}
into~\eqref{eq:finnope}
yields the following with 
an
implied constant depending on $\beta$ and
$B$,
\begin{align*}
\sum_{\substack{p_1+p_2+p_3=n
\\
\forall i:\ \F_{p_i}^*=\langle a_i \rangle}}
\prod_{i=1}^3
 \log p_i 
&=
\frac{n^2}{2}
\hspace{-0,3cm}
\sum_{\substack{\b{k} \in \N^3\\
p| k_1k_2k_3
\Rightarrow
p\leq \xi_1(n)
}}
\hspace{-0,5cm}
\Big(
\prod_{i=1}^3
\mu(k_i)
\Big) 
\mathfrak{S}_{\b{a},\b{k}}(n)
+
O
\Bigg(
\frac{n^2}{\xi_1^{\beta}}
+
\frac{n^2}{(\log n)^B}
\Big(
\sum_{\substack{\b{k} \in \N^3
\\
\forall i:\ k_i \leq \e^{2\xi_1(n)}
}}
\hspace{-0,3cm}
f_{\b{a}}(\b{k})
\Big)
\Bigg)
\\
&=
\frac{n^2}{2}
\hspace{-0,3cm}
\sum_{\substack{\b{k} \in \N^3\\
p| k_1k_2k_3
\Rightarrow
p\leq \xi_1(n)
}}
\hspace{-0,5cm}
\Big(
\prod_{i=1}^3
\mu(k_i)
\Big) 
\mathfrak{S}_{\b{a},\b{k}}(n)
+
O
\Bigg(
\frac{n^2}{\xi_1^{\beta}}
+
\frac{n^2}{(\log n)^{B/2}} 
\Bigg)
.\end{align*}
An application of
Lemma~\ref{lem:completion}
with $\epsilon=1-\beta$
shows that
\[
\sum_{\substack{p_1+p_2+p_3=n
\\
\forall i:\ \F_{p_i}^*=\langle a_i \rangle}}
\prod_{i=1}^3
 \log p_i 
-\frac{1}{2}
\bigg(
\sum_{\b{k}\in \N^3}
\mu(k_1)
\mu(k_2)
\mu(k_3)
\mathfrak{S}_{\b{a},\b{k}}(n)
\bigg)
n^2
\ll_{\beta,B}
\frac{n^2}{\min\{(\log n)^{B/2},\xi_1(n)^\beta\}},
\]
and
the proof of Theorem \ref{thm:main}
is concluded upon invoking~\eqref{eq:obsobs1}, up to the assertion that $\consta(n)\gg_a1$ whenever $\consta(n)>0$. This follows immediately from Theorem \ref{thm:const}, proved in \S \ref{s:local2glob}. Moreover, we have confirmed the shape of $\consta(n)$ given in \eqref{eq:const_1}.
\qed

\smallskip
Note that the reason for the non-explicit error term in Theorem \ref{thm:main}
is that 
the function $\xi_1$ in~\eqref{eq:obsobs0} is not explicit.

\subsection{The proof of Theorem \ref{thm:main_2}}
\label{proofofee:ggrrhh11}
Let $\beta$ be any real number in $(0,1)$ and
define
\[
\xi_1(n):=\frac{\log n}{\log \log n}
.\]
Injecting 
Proposition~\ref{prop:grhvino}
into~\eqref{eq:finnope}
provides us with
\[
\sum_{\substack{p_1+p_2+p_3=n
\\
\forall i:\ \F_{p_i}^*=\langle a_i \rangle}}
\prod_{i=1}^3
 \log p_i 
-
\frac{n^2}{2} 
\hspace{-0,1cm}
\sum_{p| k_1k_2k_3
\Rightarrow
p\leq \xi_1
}
\hspace{-0,5cm}
\mathfrak{S}_{\b{a},\b{k}}(n)
\prod_{i=1}^3
\mu(k_i)
\ll_\beta
\frac{n^2}{\xi_1^{\beta}}
+
\frac{(\log n)^6}{ n^{-11/6}  }
\Big(
\sum_{\substack{ k\in \N\\
p|k
\Rightarrow
p\leq \xi_1
}}
\hspace{-0,4cm}
k^{6}
|\mu(k)|
\Big)^3
.\]
For $n\gg 1$, each $k$ in the sum
satisfies
$k\leq \prod_{p\leq \xi_1}p\leq n^{\frac{2}{\log \log n}}$,
hence the cube of the sum 
over
$k$ 
is at most 
$n^{\frac{\theta}{\log \log n}}$ for some absolute positive
constant $\theta$.
This shows that the right side above is 
$\ll_\beta n^2 \xi_1^{-\beta}$.
Appealing to Lemma~\ref{lem:completion} completes the proof of Theorem \ref{thm:main_2}.
\qed

\section{Artin's factor
 for ternary Goldbach}
\label{s:local2glob}

In this section, we study in detail the leading factor $\consta(n)$ in Theorems \ref{thm:main} and \ref{thm:main_2}, and thus prove Theorem \ref{thm:const}, Corollary \ref{cor:loc_glob} and Theorem \ref{propositionjoe}. Recall that we have already confirmed the equality \eqref{eq:const_1} in the proof of Theorem \ref{thm:main} in \S \ref{proofofee:ggrrhh00}.
 
\subsection{The proof of~\eqref{ee:grhtyuiop}}
\label{proofofee:grhtyuiop}

Recall the definitions of $F_{a,q,k}(b)$ and $c_{a,q,k}(b)$ from the start of \S\ref{s:results}. 
It was shown by
Lenstra~\cite[Th.(3.1),Eq.(2.15)]{MR0480413}
conditionally
under $\HRH(a)$,
that for all integers $b$ and 
$q>0$
the Dirichlet density of the primes $p$
satisfying the following conditions exists,
\[
\F_{p}^*=\langle a \rangle
\text{ and }
p\equiv b \md{q},
\]
and, furthermore,
that it equals
$
\sum_{\substack{k\in \N}} 
\mu(k)
c_{a,q, k}(b)
[F_{a,q, k} : \mathbb{Q}]^{-1}
$.
This topic was later revisited by
Moree~\cite{moreeprog},
who showed that
\begin{equation}
\label{eq:len987}
\sum_{\substack{k\in \N}} 
\frac{\mu(k)c_{a,q, k}(b)}
{[F_{a,q, k} : \mathbb{Q}]}
=
\delta_a(b\hspace{-0,3cm}\mod{q})
,\end{equation}
where
$\delta_a(b\hspace{-0,2cm}\mod{q})$
is the arithmetic function
given
in Definition~\ref{def:deltaa}.
We will make consistent use of Moree's result in this section.
\begin{lemma}
\label{lem:lenlen1}
We have 
\[
\sum_{\b{k}\in \N^3}
\mu(k_1)
\mu(k_2)
\mu(k_3)
\mathfrak{S}_{\b{a},\b{k}}(n)
=
\sum_{q=1}^\infty
\sum_{c \in (\Z/q\Z)^*}
\e_q(-n c) 
\prod_{i=1}^3
\Bigg(
\sum_{b_i \in \Z/q\Z}
\e_q(b_ic)
\delta_{a_i}(b_i\hspace{-0,3cm}\mod{q}) 
\Bigg)
.\]
\end{lemma}
\begin{proof}  
Recall~\eqref{def:exp99}
and~\eqref{def:sing09}.
Lemma~\ref{lem:completion}
allows us
to rearrange
terms,
thus we can rewrite
the sum over $\b{k}$ in our lemma as 
\[ 
\sum_{q = 1}^\infty 
\sum_{\substack{c\in \Z/q\Z\\ \gcd(c,q)=1}}
\e_q(-cn) 
\prod_{i=1}^3
\l(
\sum_{k_i
\in \N} 
\frac{\mu(k_i)
S_{a_i,q, k_i}(c)}
{[F_{a_i,q, k_i} : \mathbb{Q}]}
\r)
.\]
By~\eqref{def:exp99}
the sum over $k_i$
equals 
\[ 
\sum_{\substack{b_i\in \Z/q\Z\\ \gcd(b_i,q)=1}}
\e_q(b_ic)
\sum_{k_i \in \N}
\frac{\mu(k_i)
c_{a_i,q,k_i}(b_i) 
}
{[F_{a_i,q,k_i} : \mathbb{Q}]}
\]
and using~\eqref{eq:len987}
concludes our proof.   
\end{proof}
The difficulty of converting the sum over $\b{k}$ in
\eqref{eq:const_1}
into a product
comes from the fact that the terms
$\delta_{a_i}(b_i\hspace{-0,2cm}\mod{q})$
in Lemma~\ref{lem:lenlen1} are not a multiplicative function of $q$.
These terms would be multiplicative 
in the
classical Vinogradov setting, where one has 
$\mathbf{1}_{\gcd(b_i,q)=1}(b_i)/\phi(q)$
in place of 
$\delta_{a_i}(b_i\hspace{-0,2cm}\mod{q})$.

 For brevity, we will write from now on $\beta_i(q)$ and $\Delta_i$ for $\beta_{a_i}(q)$ and $\Delta_{a_i}$.

\begin{lemma}
\label{lem:altars of madness}
If the odd part of a positive integer $q$ 
is not square-free then the following
expression 
vanishes, 
\[
\prod_{i=1}^3
\Bigg(
\sum_{b_i \in \Z/q\Z}\e_q(b_ic)\delta_{a_i}(b_i\hspace{-0,3cm}\mod{q}) 
\Bigg)
.\]
Furthermore, 
the expression vanishes 
if $\nu_2(q)  
>\min\{\nu_2(\Delta_{i}):i=1,2,3\}
$. 
\end{lemma}
\begin{proof} 
In the present proof we write $[P]:=1$ if a proposition $P$ holds, and $[P]:=0$ otherwise. 
For $1\leq i\leq 3$, we factorise each positive integer $q$ as 
$
q=
q_{i,0}
q_{i,1}
$,
where the positive integers 
$
q_{i,0},
q_{i,1}$
are uniquely 
defined through the conditions
$
p \mid q_{i,0}
\Rightarrow p|\Delta_{i}$ and  $\gcd(q_{i,1},\Delta_{i})=1$.
Now owing to 
Definition~\ref{def:deltaa}
the quantity 
$
\delta_{a_i}(b_i\hspace{-0,2cm}\mod{q})
/
\c{A}_{a_i}
$
equals
\begin{align*}
&
\Bigg(
\big[\gcd(b_i,q_{i,1})\gcd(b_i-1,q_{i,1},{h_{a_i}})=1\big]
\frac{f^\dagger_{a_i}(q_{i,1})}{\phi(q_{i,1})}
\hspace{-0.2cm}\prod_{p|b_i-1, p| q_{i,1}}\hspace{-0,2cm}\Big(1-\frac{1}{p}\Big) 
\Bigg)
\Bigg(
\frac{f_{i}^\dagger(q_{i,0})}{\phi(q_{i,0})}
\prod_{p|b_i-1, p| q_{i,0}}
\hspace{-0,2cm}
\Big(1-\frac{1}{p}\Big) 
\Bigg) 
\\
&\times 
\big[\gcd(b_i,q_{i,0})\gcd(b_i-1,q_{i,0},{h_{a_i}})=1\big]
\Bigg(
1+\left(\frac{\beta_{i}(q_{i,0})}{b_i}\right)
\mu\l(\frac{2|\Delta_i|}{\gcd(q_{i,0},\Delta_i)}\r)
f_{a_i}^\ddagger\l(\frac{|\Delta_i|}{\gcd(q_{i,0},\Delta_i)}\r) 
\Bigg)
.\end{align*} 
The integers 
$q_{i,0}$ and
$q_{i,1}$
are coprime,
hence we may write $b_i=
q_{i,0}
b_{i,1}
+
q_{i,1}
b_{i,0}
$
and use
the Chinese remainder theorem to write the sum over $b_i$
in the lemma 
as the product of 
\[  
\c{A}_{a_i}\cdot\frac{f_{a_i}^\dagger(q_{i,0})}{\phi(q_{i,0})}
\frac{f^\dagger_{a_i}(q_{i,1})}{\phi(q_{i,1})}
\sum_{
\substack{
b_{i,1}
\md{q_{i,1}}
\\
\gcd(b_{i,1},q_{i,1})=1
\\
\gcd(b_{i,1}q_{i,0}-1,q_{i,1},{h_{a_i}})=1
}}
\hspace{-0,3cm}\e(b_{i,1}c/q_{i,1})
\prod_{p|(b_{i,1}q_{i,0}-1,q_{i,1})}\hspace{-0,2cm}\Big(1-\frac{1}{p}\Big) 
\]
and
\[ 
\sum_{
\substack{
b_{i,0}\md{q_{i,0}}
\\
\gcd(
b_{i,0}
,q_{i,0})=1
\\
\gcd(
b_{i,0}q_{i,1}
-1,q_{i,0},{h_{a_i}})=1
}}
\hspace{-0,9cm}
\frac{\e(b_{i,0}c/q_{i,0})}{
\prod_{p|(b_{i,0}q_{i,1}-1,q_{i,0})} 
(1-\frac{1}{p})^{-1} 
}
\Bigg(
1+\left(\frac{\beta_i(q_{i,0})}{b_{i,0}q_{i,1}}\right)
\mu\l(\frac{2|\Delta_i|}{\gcd(q_{i,0},\Delta_i)}\r)
f_{a_i}^\ddagger \l(\frac{|\Delta_i|}{\gcd(q_{i,0},\Delta_i)}\r)  
\Bigg)
.\]
To study the sum over
$b_{i,1}$
we 
use 
Lemma~\ref{lem:exceptions}
with
\[
Q:={q_{i,1}},\quad 
r:=\prod_{p|{q_{i,1}}}p,\quad 
f(b):=
[\gcd(b,r)\gcd(b-1,r,h_{a_i})=1]
\prod_{p|b-1, p|r}\hspace{-0,2cm}\Big(1-\frac{1}{p}\Big) 
\]
to deduce that if the expression in our
lemma 
is non-vanishing then 
for each $i$
the integer 
${q_{i,1}}$
must be square-free.
Now assume that the prime $p$
satisfies 
$p\nmid \gcd(\Delta_{1},\Delta_{2},\Delta_{3})$.
Then 
there exists $i \in \{1,2,3\}$ such that
$p\nmid \Delta_{i}$
and then the
non-vanishing 
of the expression in the lemma implies that 
${q_{i,1}}$
must be square-free,
thus 
$\nu_p(q)=\nu_p({q_{i,1}})
\leq 1$. 

Now the sum over
$b_{i,0}$
can be 
studied via 
Lemma~\ref{lem:exceptions}
with 
$
Q:={q_{i,0}}$,
$r:=\gcd(
{q_{i,0}}
,\Delta_{i})
$
and with $f(b)$ being the product of
$
[\gcd(b,r)\gcd(b{q_{i,1}}-1,r,h_{a_i})=1]
$
and
\[
\Bigg\{
1+\left(\frac{\beta({q_{i,0}})}{b}\right)
\mu\l(\frac{2|\Delta_{i}|}{\gcd({q_{i,0}},\Delta_{i})}\r)
f_{i}^\ddagger\l(\frac{|\Delta_i|}{\gcd({q_{i,0}},\Delta_i)}\r) 
\Bigg\}
\prod_{p|(b{q_{i,1}}-1, r)}\hspace{-0,2cm}\Big(1-\frac{1}{p}\Big) 
.\]
We have
used the fact that 
$p\mid q_{i,0}\Leftrightarrow p\mid r$ 
and that the Kronecker symbol has period 
$|\beta({q_{i,0}})|=r$.
Lemma~\ref{lem:exceptions} shows
that unless the expression in
our lemma vanishes, we have 
$
\gcd({q_{i,0}},\Delta_{i})
=
{q_{i,0}}
$,
thus 
for every $i$ 
we must have 
${q_{i,0}} \mid \Delta_{i}$.
Now if a prime $p$ satisfies
$p\mid \gcd(\Delta_{1},\Delta_{2},\Delta_{3})$
we have that 
for every $i$,
$
\nu_p(q)=
\nu_p({q_{i,0}})
\leq 
\nu_p(\Delta_{i})
$,
thus
$\nu_p(q)\leq \min\{\nu_p(\Delta_{i}):i=1,2,3\}$.
If $p\neq 2$ then 
this shows that $\nu_p(q)\leq 1$ since the odd part of a fundamental discriminant is square-free,
while if $p=2$ then we must have 
$\nu_2(q)\leq 
\min\{\nu_2(\Delta_{i}):i=1,2,3\}
$.
\end{proof}

Lemma~\ref{lem:altars of madness}
allows us to simplify the summation over $q$ in 
Lemma~\ref{lem:lenlen1}
since the only integers $q$ making a contribution towards the sum 
must satisfy 
\[
\forall p, i:\  p|\Delta_{i},p|q\Rightarrow \nu_p(q)\leq \nu_p(\Delta_{i})
\quad\text{ and }\quad
p|q, p\nmid \Delta_{1}\Delta_{2}\Delta_{3}
\Rightarrow 
\nu_p(q)\leq 1
.\]
To keep track of every factorisation we 
introduce
for every $q\in \N$ and $\b{w} \in \{0,1\}^3$
the positive integer 
\[
q(\b{w})
:=\prod_{\substack{p:\\ \forall i : \ 
p \mid \Delta_{i}
\Leftrightarrow
\b{w}(i)=0
}}p^{\nu_p(q)}
\]
so that 
$
q=\prod_{\b{w} \in \F_2^3}
q(\b{w})$.
Furthermore, 
whenever
$
\b{w}
\neq \b{u}
$ then 
we have 
$\gcd(q(\b{w}),q(\b{u}))=1$.
Note that for a 
given $q$,
$q(\b{w})$ 
 is uniquely characterised by the properties
\begin{equation}
\label{eq:sumprop}
\gcd(q(\b{w}),\prod_{i:\b{w}(i)=1}\Delta_{i})=1
\quad\text{ and }\quad
q(\b{w})
\mid \gcd \{\Delta_{i}: \b{w}(i)=0\}
.\end{equation}
 In the case $\b{w}=(1,1,1)$, the latter condition is interpreted as vacuous.
It may be that for certain values of  $a_i$ 
and for all $q$
some $q(\b{w})$ are equal to $1$;
for example,
this happens if $a_1=a_2=a_3$,
in which case we have 
$\b{w}\notin \{(0,0,0),(1,1,1)\}\Rightarrow q(\b{w})=1$.
We now use the 
definition of $q(\b{w})$,
Lemma~\ref{lem:lenlen1}
and Lemma~\ref{lem:altars of madness}
to infer
\begin{align}
\label{eq:starbg}  
\begin{split}
\sum_{\b{k}\in \N^3}
\mu(k_1)
\mu(k_2)
\mu(k_3)
\mathfrak{S}_{\b{a},\b{k}}(n)
&
=
\sum_{\substack{ (q(\b{w}))\in \N^8,
\\
\eqref{eq:sumprop}
\text{ holds}
\\
\mu(q((1,1,1)))^2=1
}}
\sum_{\substack{
c\md{\prod_{\b{w}}
q(\b{w})
}
\\
\gcd(c,
\prod_{\b{w}}q(\b{w}))=1
}}
\e(-n c\prod_{\b{w}}q(\b{w})^{-1}) 
\times \\ &
\times 
\prod_{i=1}^3
\Bigg(
\sum_{b_i \md{\prod_{\b{w}}q(\b{w})
}}
\e\Big(b_ic \prod_{\b{w}}
q(\b{w})^{-1}\Big)
\delta_{a_i}
\Big(b_i\hspace{-0,3cm}\mod{\prod_{\b{w}}
q(\b{w})
}\Big) 
\Bigg)
.\end{split}
\end{align}
Noting that 
the integers 
$\prod_{\b{w}(i)=0}q(\b{w})$ and
$\prod_{\b{w}(i)=1}q(\b{w})$
are coprime,
that 
\[
\gcd\Big(\Delta_{i},\prod_{\b{w}}
q(\b{w})
\Big)
=
\prod_{\b{w}(i)=0}
q(\b{w})
\]
and 
recalling Definition~\ref{def:deltaa}
we see that 
\[
\delta_{a_i}
\Big(b_i\hspace{-0,3cm}\mod{
\prod_{\b{w}}
q(\b{w})
}\Big) 
=
\delta_{a_i}
\Big(b_i\hspace{-0,3cm}\mod{
\prod_{\b{w}(i)=0}q(\b{w})
}\Big) 
\c{A}_{a_i}
\Big(b_i\hspace{-0,3cm}\mod{
\prod_{\b{w}(i)=1}q(\b{w})
}\Big) 
\c{A}_{a_i}^{-1}
.\]
Writing 
$
b_i
=
b'_i
\prod_{\b{w}(i)=1}
q(\b{w})
+
b''_i
\prod_{\b{w}(i)=0}q(\b{w})
$
and using the Chinese remainder theorem we obtain 
\begin{align*}
&\sum_{b_i \md{\prod_{\b{w}}
q(\b{w})
}}
\e\Big(b_ic \prod_{\b{w}}q(\b{w})^{-1}\Big)
\delta_{a_i}
\Big(b_i\hspace{-0,3cm}\mod{\prod_{\b{w}}q(\b{w})}\Big)
\\=&
\sum_{b'_i \md{\prod_{\b{w}(i)=0}q(\b{w})}}
\e\Big(b'_ic \prod_{\b{w}(i)=0}q(\b{w})^{-1}\Big)
\delta_{a_i}
\Big(b'_i \prod_{\b{w}(i)=1}q(\b{w})\hspace{-0,3cm}\mod{\prod_{\b{w}(i)=0}q(\b{w})}\Big)
\times
\\
\times &\sum_{b''_i \md{\prod_{\b{w}(i)=1}q(\b{w})}}
\e\Big(b''_ic \prod_{\b{w}(i)=1}q(\b{w})^{-1}\Big)
\c{A}_{a_i}^{-1}
\c{A}_{a_i}
\Big(b''_i \prod_{\b{w}(i)=0}q(\b{w}) \hspace{-0,3cm}\mod{\prod_{\b{w}(i)=1}q(\b{w})}\Big)
.\end{align*}
For the further analysis of the expressions above, we introduce 
for 
$r\in \N$,
$c\in \Z$
the quantity
\begin{equation}
\label{def:bombpi}
\c{M}_{a}(c,r)
:=
\frac{1}{\c{A}_a}
\sum_{b\md{r}}
\e_r(bc)
\c{A}_{a}(b\hspace{-0,3cm}\mod{r}),
\end{equation}
and for 
$\b{r}\in \N^k$,
$\b{c}\in \Z^k$ 
define 
\[
\c{D}_{a}(\b{c},\b{r})
:=\sum_{
b \md{r_1\cdots r_k}
}
\e\Big[b \Big(\sum_{i=1}^r
\frac{c_i}{r_i}\Big)
\Big]
\delta_{a}(b\hspace{-0,3cm}\mod{r_1\cdots r_k})
.\]
Hence, writing 
\[
c=\sum_{\b{w}\in \{0,1\}^3}
c^{[\b{w}]}
\prod_{\b{v}\neq \b{w}}q(\b{v}),
\]
we see that 
$
\prod_{\b{w}(i)=1} 
\c{M}_{a_i}(c^{[\b{w}]},q(\b{w}))
$
equals
\[
\c{A}_{a_i}^{-1}
\sum_{b''_i \md{\prod_{\b{w}(i)=1}q(\b{w})}}
 \hspace{-0,2cm}
\e\Big(b''_ic \prod_{\b{w}(i)=1}q(\b{w})^{-1}\Big)\c{A}_{a_i}\Big(b''_i \prod_{\b{w}(i)=0}q(\b{w}) \hspace{-0,3cm}\mod{\prod_{\b{w}(i)=1}q(\b{w})}\Big)
\]
and 
that 
$
\c{D}_{a_i}((c^{[\b{w}]})_{_{\b{w}(i)=0}}, (q(\b{w}))_{_{\b{w}(i)=0}}) 
$
is 
\[
\sum_{b'_i \md{\prod_{\b{w}(i)=0}q(\b{w})}}
 \hspace{-0,2cm}
\e\Big(b'_ic \prod_{\b{w}(i)=0}q(\b{w})^{-1}\Big)
\delta_{a_i}
\Big(b'_i \prod_{\b{w}(i)=1}q(\b{w})\hspace{-0,3cm}\mod{\prod_{\b{w}(i)=0}q(\b{w})}\Big) 
.\]

Let us bring into play the entities
\[
\Delta_\b{w}
:=
\prod_{p\nmid \prod_{\b{w}(i)=1}\Delta_i}
p^{\min\{\nu_p(\Delta_i)\ :\ \b{w}(i)=0\}},
\]
which we interpret as $1$ in case $\b{w}=(1,1,1)$, 
and note that 
$\prod_\b{w}  
\Delta_\b{w}$
coincides with 
the entity
$\Da$
introduced in~\eqref{def:mytonga}.
We 
see that the sum in~\eqref{eq:starbg}
becomes  
\begin{align*}
\sum_{\substack{(q(\b{w})) \in \N^8
\\
\b{w}\neq (1,1,1)
\Rightarrow 
q(\b{w})
 \mid  {\Delta_\b{w}}
\\
\mu(q((1,1,1)))^2=1\\
\gcd(q((1,1,1)),\Delta_{1}\Delta_{2}\Delta_{3})=1
}}
\hspace{-0,3cm}
&\sum_{
(c^{[\b{w}]})
\in \prod_\b{w} (\Z/q(\b{w})\Z)^*
}
\hspace{-0,2cm}
\Big(\prod_{\b{w}}
\e_{q(\b{w})}(-n c^{[\b{w}]})
\Big)
\times
\\
&
\times
\prod_{i=1}^3
\Big\{
\c{D}_{a_i}((c^{[\b{w}]})_{_{\b{w}(i)=0}}, (q(\b{w}))_{_{\b{w}(i)=0}})
\hspace{-0,2cm}
\prod_{\b{w}(i)=1}
\c{M}_{a_i}(c^{[\b{w}]},q(\b{w}))
\Big\}
.\end{align*}
Clearly, the terms corresponding to $q((1,1,1))$ can be separated,
thus, in light of~\eqref{eq:starbg},
we are led to 
\begin{equation}
\label{eq:breakas8}
\sum_{\b{k}\in \N^3}
\mu(k_1)
\mu(k_2)
\mu(k_3)
\mathfrak{S}_{\b{a},\b{k}}(n)
=
S_{\b{a},0}(n)
S_{\b{a},1}(n)
,\end{equation}
where
\begin{align*}
S_{\b{a},0}(n)
:= 
\sum_{\substack{(q(\b{w}))_{\b{w}\neq (1,1,1)
} \in \N^7\\
q(\b{w}) \mid  {\Delta_\b{w}}
}}
&
\sum_{
(c^{[\b{w}]})
\in \prod_{\b{w}\neq (1,1,1)} (\Z/q(\b{w})\Z)^*
}
\Big(\prod_{\b{w}\neq (1,1,1)
}
\e_{q(\b{w})}(-n c^{[\b{w}]})
\Big)
\times
\\
&
\times
\prod_{i=1}^3
\Big\{
\c{D}_{a_i}((c^{[\b{w}]})_{\b{w}(i)=0},(q(\b{w}))_{\b{w}(i)=0})
\prod_{\substack{\b{w}(i)=1\\\b{w}\neq (1,1,1)}}
\c{M}_{a_i}(c^{[\b{w}]},q(\b{w}))
\Big\} 
\end{align*}
and
\begin{equation}\label{eq:def_S1}
\begin{aligned}
S_{\b{a},1}(n) := \sum_{\substack{
\gcd(q((1,1,1)),\Delta_{1} \Delta_{2} \Delta_{3})=1
}}
&\mu(q((1,1,1)))^2
\times
\\
&\hspace{-2cm}\times 
\hspace{-1cm}
\sum_{
c^{[(1,1,1)]}
\in (\Z/q((1,1,1))\Z)^*
}
\e_{q((1,1,1))}(-n c^{[(1,1,1)]})
\prod_{i=1}^3
\c{M}_{a_i}(c^{[(1,1,1)]},q((1,1,1)))
.\end{aligned}
\end{equation}

\begin{lemma}\label{lem:sometimesyoucanseethecomments}
For any 
 $q\in \N$ and 
$\b{w}\in \{0,1\}^3$
define
$
 {d_\b{w}}
:= {\Delta_\b{w}}/q(\b{w})$. 
\begin{enumerate}
\item 
Let $i\in\{1,2,3\}$ and for each $\b{w}$ with $\b{w}(i)=0$ let
$
c^{[\b{w}]}
\in (\Z/q(\b{w})\Z)^*
$. Then  
\[
\c{D}_{a_i}((c^{[\b{w}]}  )_{\b{w}(i)=0},(q(\b{w}))_{\b{w}(i)=0})
=
\c{D}_{a_i}((c^{[\b{w}]} {d_\b{w}})_{\b{w}(i)=0},( {\Delta_\b{w}})_{\b{w}(i)=0})
.\]
\item
Let $i\in\{1,2,3\}$,
$\b{w} \in \{0,1\}^3\setminus \{(1,1,1)\}$ 
with 
 $\b{w}(i)=1$
and 
$
c^{[\b{w}]}
\in (\Z/q(\b{w})\Z)^*
$. 
Then 
\[
\c{M}_{a_i}(c^{[\b{w}]}, {q}(\b{w}))
=
\c{M}_{a_i}(c^{[\b{w}]} {d_\b{w}}, {\Delta_\b{w}}) 
.\]
\end{enumerate}
\end{lemma}
\begin{proof}
\emph{(1)}: 
Define 
\[Q:=
\prod_{\b{w}: \b{w}(i)=0}
q(\b{w})=\prod_{\b{w}:\b{w}(i)=0}   \frac{\Delta_\b{w}}{d_\b{w}}
\
\
\text{ and }
\
\
D:=
\prod_{\b{w}: \b{w}(i)=0}
\Delta_\b{w}
.\]
If we assume $\HRH(a_i)$
then it is immediately clear from Moree's interpretation of $\delta_{a_i}$
as Dirichlet densities~\cite{moreeprog}
that the following holds,
\[
\delta_{a_i}
\big(m\hspace{-0,3cm}\mod{Q
}\big) 
=\sum_{\substack{b\md{D}
\\
b\equiv m \md{ Q }
}}
\delta_{a_i}
\big(b\hspace{-0,3cm}\mod{D
}\big) 
.\]
One can also prove this unconditionally 
directly from Definition~\ref{def:deltaa}
via 
a  
tedious but straightforward
calculation
that we do not reproduce
here.
To conclude the proof 
we observe that 
\begin{align*}
\c{D}_{a_i}((c^{[\b{w}]}  )_{\b{w}(i)=0},(q(\b{w}))_{\b{w}(i)=0})
=&
\sum_{m\md{ 
Q 
}}
\e
\bigg(
m \sum_{\b{w}:\b{w}(i)=0} \frac{c^{[\b{w}]}}{q(\b{w})}
\bigg)
\delta_{a_i}(m\hspace{-0,3cm}\mod{ 
Q 
})
\\=&
\sum_{b\md{  D}}
e\bigg(
b \sum_{\b{w}:\b{w}(i)=0} \frac{c^{[\b{w}]}   d_\b{w}}{ \Delta_\b{w}} 
\bigg) 
\delta_{a_i}(b\hspace{-0,3cm}\mod{  D})
\\=
&\c{D}_{a_i}((c^{[\b{w}]} {d_\b{w}})_{\b{w}(i)=0},( {\Delta_\b{w}})_{\b{w}(i)=0})
.\end{align*}

 \emph{(2)}: 
Due to the assumption that $\b{w}(i)=1$ we have $\gcd(\Delta_\b{w},\Delta_{i})=1$,
and thus,
\[\frac{\c{A}_{a_i}(m\hspace{-0,3cm}\mod{ {\Delta_\b{w}}})}{\c{A}_{a_i}}= 
\frac{\delta_{a_i}(m\hspace{-0,3cm}\mod{ {\Delta_\b{w}}})}{\c{L}_{a_i}}.\]
 We similarly have 
\[\frac{\c{A}_{a_i}(m\hspace{-0,3cm}\mod{ {\Delta_\b{w}}/ {d_\b{w}}})}{\c{A}_{a_i}}= 
\frac{\delta_{a_i}(m\hspace{-0,3cm}\mod{ {\Delta_\b{w}}/ {d_\b{w}}})}{\c{L}_{a_i}}.\]
By $\HRH(a_i)$
it then follows that 
 \[
\c{A}_{a_i}(m\hspace{-0,3cm}\mod{ {\Delta_\b{w}}/ {d_\b{w}}})
=
\sum_{\substack{
b\md{
{\Delta_\b{w}}}
\\b\equiv m \md{ {\Delta_\b{w}}/ {d_\b{w}}}
}}
\c{A}_{a_i}(b\hspace{-0,3cm}\mod{ {\Delta_\b{w}}})
,\]
which can also be shown unconditionally
as above.
The rest of the proof is conducted as in the first part. 
\end{proof} 
For the analysis of $S_{\b{a},1}(n)$, we recall the definition of $\sigma_{\b{a},n}(d)$ in \eqref{def:sigmapenotdelta} and use the following lemma.
\begin{lemma}\label{lem:sigmaan}
  If $p\nmid\Delta_1\Delta_2\Delta_3$, then
  \begin{equation*}
    \sigma_{\b{a},n}(p) = 1+
\sum_{
c
\in (\Z/p\Z)^*
}
\e_{p}(-n c)
\prod_{i=1}^3
\c{M}_{a_i}(c,p)
  \end{equation*}
\end{lemma}

\begin{proof}
  The easily verified equality $\sum_{b\md{p}}\c{A}_{a_i}(b\hspace{-0,2cm}\mod{p}) =\c{A}_{a_i}$ shows that the expression on the right-hand side is equal to
\begin{align*}
\sum_{c\in \Z/p\Z}\e_{p}(-cn)
\prod_{i=1}^3
\c{M}_{a_i}(c,p)
=&
\sum_{\substack{\b{b}\in (\Z/p\Z)^3}}
\Bigg(
\prod_{i=1}^3
\frac{\c{A}_{a_i}(b_i\hspace{-0,3cm}\mod{p})}{\c{A}_{a_i}}
\Bigg)
\sum_{c\in \Z/p\Z}
\e_p(c(b_1+b_2+b_3-n))
\\
=&
p
\sum_{\substack{\b{b}\in (\Z/p\Z)^3
\\
\sum_{i=1}^3 b_i\equiv n \md{p}}}
\prod_{i=1}^3
\frac{\c{A}_{a_i}(b_i\hspace{-0,3cm}\mod{p})}{\c{A}_{a_i}}
.\end{align*}
Since $p\nmid\Delta_1\Delta_2\Delta_3$, we see that $\c{A}_{a_i}(b_i\mod{p})/\c{A}_{a_i} = \delta_{a_i}(b_i\mod{d})/\c{L}_{a_i}$.
\end{proof}

Using \eqref{eq:def_S1}, multiplicativity and Lemma \ref{lem:sigmaan}, we infer that  
\begin{equation}
\label{eq:breakas8715}
S_{\b{a},1}(n)
= 
\prod_{p\nmid 
\Delta_{1} \Delta_{2} \Delta_{3}
}
\hspace{-0,2cm}
\Big(1+
\sum_{
c
\in (\Z/p\Z)^*
}
\e_{p}(-n c)
\prod_{i=1}^3
\c{M}_{a_i}(c,p)
\Big) = \prod_{p\nmid 
\Delta_{1} \Delta_{2} \Delta_{3}
}
\sigma_{\b{a},n}(p). 
\end{equation}

We now turn our attention to
$S_{\b{a},0}(n)$.
Letting $
 {d_\b{w}}
:= {\Delta_\b{w}}/q(\b{w})$
we use 
Lemma~\ref{lem:sometimesyoucanseethecomments}
to 
obtain
\begin{align*}
S_{\b{a},0}(n) = \sum_{\substack{(d_\b{w})_{\b{w}\neq 
(1,1,1)
} \in \N^7\\
d_\b{w}
\mid  \Delta_\b{w}
}}
&\sum_{
(c^{[\b{w}]})\in \prod_\b{w\neq (1,1,1)} 
\big(
\frac{\Z}{(\Delta_\b{w}/d_\b{w})\Z}
\big)^* 
}
\Big(\prod_{\b{w}\neq (1,1,1)}
\e\Big(
-n
c^{[\b{w}]}
 d_\b{w}/
\Delta_\b{w}
\Big)
\Big)\times\\
&\times
\prod_{i=1}^3
\Big\{
\c{D}_{a_i}((c^{[\b{w}]} {d_\b{w}})_{\b{w}(i)=0},( {\Delta_\b{w}})_{\b{w}(i)=0})
\prod_{\substack{\b{w}(i)=1\\\b{w}\neq (1,1,1)}}
\c{M}_{a_i}(c^{[\b{w}]} {d_\b{w}}, {\Delta_\b{w}}) 
\Big\}
.
\end{align*} 
For any 
$ {d_\b{w}}$
with 
$ {d_\b{w}} \mid  {\Delta_\b{w}}$
the elements $y^{[\b{w}]}
\md{{\Delta_\b{w}}}$ 
that 
satisfy the condition
$\gcd(y^{[\b{w}]}, {\Delta_\b{w}})
=
 {d}_\b{w}
$ 
are   exactly those of the form 
\[
y^{[\b{w}]}=c^{[\b{w}]}  {d_\b{w}},\quad 
c^{[\b{w}]}
\in  \Big(\frac{\Z}{( {\Delta_\b{w}}/ {d_\b{w}})\Z}\Big)^*
.\] 
We thus obtain that 
the sum over $ {d_\b{w}},c^{[\b{w}]}$
equals
\begin{align*}
\sum_{
(y^{[\b{w}]})
\in 
\prod_\b{w\neq (1,1,1)} 
(
\Z/{\Delta_\b{w}}\Z
) }
&
\Big(\prod_{\b{w}\neq (1,1,1)}
\e\Big(-n
y^{[\b{w}]}
/{\Delta_\b{w}}
\Big)
\Big)
\times \\ & \times
\prod_{i=1}^3
\Big\{
\c{D}_{a_i}((y^{[\b{w}]})_{\b{w}(i)=0},( {\Delta_\b{w}})_{\b{w}(i)=0}) 
\prod_{\substack{\b{w}(i)=1\\\b{w}\neq (1,1,1)}}
\c{M}_{a_i}(y^{[\b{w}]}, {\Delta_\b{w}}) 
\Big\}
.\end{align*}
By definition, $ {\Delta_{(1,1,1)
}}=1$,
so
$
 \Da
=\prod_{\b{w}\neq (1,1,1)}
 {\Delta_\b{w}}$.
Note that 
 $\gcd( {\Delta_\b{w}},{\Delta_\b{v}})=1$ for $\b{w}\neq \b{v}$. Using the Chinese remainder theorem 
and 
writing every $y \md{\prod_{\substack{
\b{v}\neq (1,1,1)}}
 {\Delta_\b{w}}
}$
as 
\[
y=\sum_{\b{w}\neq (1,1,1)}
y^{[\b{w}]}
\prod_{\substack{
\b{v}\notin \{\b{w},(1,1,1)\}
}} 
{\Delta_\b{v}},
\]
we see that the sum over $y^{[\b{w}]}$ equals 
\[
\sum_{y\md{\Da}}
\e(-ny/
\Da
)
\prod_{i=1}^3
\Bigg(
\sum_{b_i\md{\Da}}
\e(b_iy/\Da)
\delta_{a_i}(b_i \hspace{-0,3cm} \mod{\Da})
\Bigg) 
.\]
This is clearly 
\[
\Da
\sum_{\substack{\b{b}\md{\Da
}
\\
\sum_{i=1}^3 b_i\equiv n \md{\Da}
}}
\prod_{i=1}^3 
\delta_{a_i}(b_i \hspace{-0,3cm} \mod{\Da})
,\]
thus,
recalling~\eqref{def:sigmapenotdelta},
we have shown that  
\begin{equation}
\label{eq:breakas871b}
S_{\b{a},0}(n)
=
\sigma_{\b{a},n}(\Da)
\prod_{i=1}^3\c{L}_{a_i}
.\end{equation}

The proof of~\eqref{ee:grhtyuiop}
is concluded upon combining~\eqref{eq:breakas8},~\eqref{eq:breakas8715}
and~\eqref{eq:breakas871b}.

\subsection{The proof of~\eqref{th:grhreplaceall}}
\label{proofofth:grhreplaceall}
We begin
by  
finding an explicit expression for
$\sigma_{\b{a},n}(p)$, for $p\nmid\Delta_1\Delta_2\Delta_3$, 
that is explicit in terms of $n$ and the $h_{a_i}$.
Define 
\[
\theta_a(p):=
\begin{cases}
1,
&\text{ if } p\mid h_a,\\  
\frac{1}{p},
&\text{ if } p\nmid h_a.\\  
\end{cases}
\]
\begin{lemma}
\label{lem:pila}
For an integer $c$ 
and a prime $p$
with  
$p\nmid c$ 
we have 
\[
\c{M}_{a}(c,p)=
-
\frac
{(1+\theta_a(p)\e_p(c))}
{(p-1-\theta_a(p))}
.\]
\end{lemma}
\begin{proof}
Combining~\eqref{def:prolabainoume} and~\eqref{def:bombpi}
we immediately infer 
\[
\c{M}_{a}(c,p)=
\frac{1}{(p-1-\theta_a(p))}
\sum_{\substack{b\md{p}\\ \gcd(b,p)=1\\ \gcd(b-1,p,h_a)=1}}
\e_p(bc)
\prod_{\substack{\ell \text{ prime}\\ \ell \mid \gcd(b-1,p)}} \Big(1-\frac{1}{\ell}\Big)
.\]
It is now 
easy to see that 
the sum over $b$ equals 
$-1-\e_p(c)$ 
or
$-1-\e_p(c)/p$
according to whether
$p\mid h_a$
or
$p\nmid h_a$. 
\end{proof}
Let us 
denote the elementary symmetric polynomials in $\theta_{a_i}(p)$ by
\begin{align*} 
&\Xi_0(p):= 1,\\
&\Xi_1(p):= \theta_{a_1}(p)+\theta_{a_2}(p)+\theta_{a_3}(p),\\
&\Xi_2(p):= \theta_{a_1}(p)\theta_{a_2}(p)+\theta_{a_2}(p)\theta_{a_3}(p)+\theta_{a_1}(p)\theta_{a_3}(p),\\
&\Xi_3(p):=\theta_{a_1}(p)\theta_{a_2}(p)\theta_{a_3}(p).
\end{align*}
\begin{lemma}
\label{lem:anapantexo}
For every odd integer $n$ and prime $p\nmid 
\prod_{i=1}^3
\Delta_{i}
$ 
we have 
\[  
\sigma_{\b{a},n}(p)
=
1
-
\frac{p}{\prod_{1\leq i \leq 3}{(p-1-\theta_{a_i}(p))}}
\Big(
\sum_{\substack{0\leq j \leq 3\\ j \equiv n \md{p}}}\Xi_j(p)
\Big)
+
\prod_{1\leq i \leq 3}
\Big(
\frac{1+\theta_{a_i}(p)}{
p-1-\theta_{a_i}(p)}
\Big)
.\]
\end{lemma}
\begin{proof}
By Lemma~\ref{lem:sigmaan}
and Lemma~\ref{lem:pila}
we see that 
\[
\sigma_{\b{a},n}(p)
=1-\frac{1}{\prod_{1\leq i \leq 3}{(p-1-\theta_{a_i}(p))}}
\sum_{c \in (\Z/p\Z)^*}\e_p(-cn)
\prod_{1\leq i \leq 3}
{(1+\theta_{a_i}(p)\e_p(c))}
.\]
The sum 
over $c$
equals
\[
\sum_{0\leq j \leq 3}
\Xi_j(p)
\sum_{c \in (\Z/p\Z)^*}\e_p(c(j-n))
=
p
\Big(
\sum_{\substack{0\leq j \leq 3\\ j \equiv n \md{p}}}\Xi_j(p)
\Big)
-\prod_{1\leq i \leq 3}(1+\theta_{a_i}(p))
\]
and the proof is complete. 
\end{proof}

\begin{lemma}
\label{lem:nmiddeducomeon2}
Let $n$ be an odd integer. If $3\mid\Delta_1\Delta_2\Delta_3$, then $\prod_{p\nmid \Delta_1\Delta_2\Delta_3}\sigma_{\b{a},n}(p) \neq 0$. If $3\nmid\Delta_1\Delta_2\Delta_3$, then the following are equivalent: 
\begin{enumerate}
\item $\prod_{p\nmid \Delta_1\Delta_2\Delta_3}\sigma_{\b{a},n}(p) = 0$,
\item $\sigma_{\b{a},n}(3) =0$,
\item One of the following two conditions holds,
\begin{align*}
&3 \text{ divides every element in the set } 
\{
h_{a_1},
h_{a_2},
h_{a_3}
\}
\text{ and }
3\nmid n,\quad \text{ or }\\
&3 \text{ divides exactly two elements in the set } \{
h_{a_1},
h_{a_2},
h_{a_3}
\},
\text{ and }
n\equiv 1 \md{3}
.
\end{align*}
\end{enumerate}
Furthermore, 
$\prod_{p\nmid 
\Delta_{1}
\Delta_{2}
\Delta_{3}
}
\sigma_{\b{a},n}(p)
\neq 0$
implies
$\prod_{p\nmid
\Delta_{1}
\Delta_{2}
\Delta_{3}
}
\sigma_{\b{a},n}(p)\gg 1$, with an absolute implied constant.
\end{lemma}

\begin{proof}
For a prime
$p\nmid 
\Delta_{1}\Delta_{2}
\Delta_{3}
$ 
with $p\geq 5$
there exists 
at most one
 $ 0\leq j \leq 3$
satisfying
$j\equiv n \md{p}$,
therefore 
$$
\sum_{\substack{0\leq j \leq 3\\ j \equiv n \md{p}}}\Xi_j(p)
\leq 3
.$$
Invoking Lemma~\ref{lem:anapantexo}
we obtain  
\[
\sigma_{\b{a},n}(p)
>1-\frac{3p}{(p-2)^3}
+\frac{1}{(p-1)^3}
.\]
Recall
that 
no $a_i$ is a square,
hence 
$2\nmid h_{a_1}h_{a_2}h_{a_3}$.
The fact that $n$ is odd implies that   
\[
\sum_{\substack{0\leq j \leq 3\\ j \equiv n \md{2}}}\Xi_j(2)
=
\Xi_1(2)+
\Xi_3(2)
=
\frac{13}{8}
,\]
hence
if $\Delta_{1}
\Delta_{2}
\Delta_{3}$
is odd 
we
can use 
Lemma~\ref{lem:anapantexo}
to show that 
$\sigma_{\b{a},n}(2)=2$. 
We have shown that for odd $n$ one has 
\[
\prod_{\substack{p\nmid \Delta_{1}
\Delta_{2}
\Delta_{3}
\\ p\neq 3}} 
\sigma_{\b{a},n}(p)
\gg 1
\]
with an absolute implied constant
and it remains to study 
$\sigma_{\b{a},n}(3)$.
One can find an explicit formula for this density
by fixing the congruence class of $n\md{3}$.
For example, in the case that
 $n\equiv 1 \md{3}$
we have
\[
\sigma_{\b{a},n}(3)
=
1
-
\frac{3(
\theta_{a_1}(3)
+
\theta_{a_2}(3)
+
\theta_{a_3}(3)
)}{\prod_{1\leq i \leq 3}{(2-\theta_{a_i}(3))}}
+
\prod_{1\leq i \leq 3}
\Big(
\frac{1+\theta_{a_i}(3)}{
2-\theta_{a_i}(3)}
\Big)
\]
and we can check that  
$\sigma_{\b{a},n}(3)=0$
if and only if 
at most one of the $\theta_i$ is equal to $1/3$.
A case by case analysis reveals that if $n\equiv 2 \md{3}$
then
$\sigma_{\b{a},n}(3)=0$
if and only if 
$(\theta_{a_i}(3))_i=(1,1,1)$
and that 
 if $n\equiv 0 \md{3}$
then
$\sigma_{\b{a},n}(3)$
never vanishes.
Noting that 
$\sigma_{\b{a},n}(3)$
attains only finitely many values
as it only depends on 
$n\md{3}$
and the choice of 
$(\theta_{a_i}(3))_i \in \{1,\frac{1}{3}\}^3$,
we see that there exists an absolute constant $c$
such that 
if 
$\sigma_{\b{a},n}(3)>0$
then 
$\sigma_{\b{a},n}(3)>c,$
thus concluding our proof.
\end{proof}
We next provide
a lower 
bound
for
$S_{\b{a},0}(n)$, see \eqref{eq:breakas871b}.
One could proceed
by finding
explicit expressions,
however,
this will lead to rather
more complicated 
formulas than the one
for $S_{\b{a},1}(n)$
in Lemma~\ref{lem:anapantexo}.
We shall instead opt
to bound the densities 
$\delta_a(b_i\hspace{-0,2cm}\mod{\Da})$
from below 
in~\eqref{eq:breakas871b}
and then count the number of solutions of the equation
$n\equiv x_1+x_2+x_3\md{\Da}$
such that 
for every $i$ we have 
$\delta_a(x_i\hspace{-0,2cm}\mod{\Da})\neq 0$. 
\begin{lemma}
\label{lem:nonzero}
For any integers
$q$ and $x$
such that $q$ is positive and  
$\delta_a(x\hspace{-0,2cm}\mod{q})>0$
we have
\[
\delta_a(x\hspace{-0,3cm}\mod{q}) 
\gg
\frac{\phi(h_a)}{qh_a}
,\]
with an absolute implied constant.
\end{lemma}
\begin{proof}
Under the assumptions of our lemma
we have the following due to
Definition~\ref{def:deltaa},
\[
\delta_a(x\hspace{-0,3cm}\mod{q})
\c{A}_a^{-1}
\frac{\phi(q)}{f_a^\dagger(q)}
\prod_{p|x-1, p| q}
\hspace{-0,2cm}
\Big(1-\frac{1}{p}\Big)^{-1}
=
1+\mu\l(\frac{2|\Delta_a|}{\gcd(q,\Delta_a)}\r)\left(\frac{\beta_a(q)}{x}\right)f_a^\ddagger\l(\frac{|\Delta_a|}{\gcd(q,\Delta_a)}\r) 
.\]
The right-hand side is either $\geq 1$ or equal to $1-f_a^\ddagger(|\Delta_a|\gcd(q,\Delta_a)^{-1})$. In the latter case, since the right-hand side must be positive and $f_a^\ddagger(|\Delta_a|\gcd(q,\Delta_a)^{-1})^{-1}$ is an integer, we see that the right-hand side is $\geq 1/2$. Therefore, under the assumptions 
of our lemma we have 
\[
\delta_a(x\hspace{-0,3cm}\mod{q})
\geq
\frac{\c{A}_a}{2}
\frac{f_a^\dagger(q)}{\phi(q)}
\prod_{p|x-1, p| q}
\hspace{-0,2cm}
\Big(1-\frac{1}{p}\Big)
.\]
It is obvious that  
$
\c{A}_a
f_a^\dagger(q)
\gg
\phi(h_a)/h_a
,$
with an implied absolute constant.
This is sufficient for our lemma
owing to 
$
\prod_{p|x-1, p|q}
(1-\frac{1}{p})
\geq 
\phi(q)/q
$.
 \end{proof} 
Recalling~\eqref{def:sigmapenotdelta}
we see that 
\[
\sigma_{\b{a},n}(\Da)
\prod_{i=1}^3\c{L}_{a_i}
=
\Da
\sum_{\substack{b_1,b_2,b_3 \md{\Da}\\ 
b_1+b_2+b_3\equiv n \md{\Da}
}} 
\prod_{i=1}^3 
\delta_{a_i}(b_i\hspace{-0,3cm}\mod{\Da})  
,\]
thus, 
if $\sigma_{\b{a},n}(\Da)>0$
then there exist 
$x_1,x_2,x_3\md{\Da}$
such that 
$\prod_{i=1}^3 
\delta_{a_i}(x_i\hspace{-0,2cm}\mod{\Da})  
>0$
and 
$x_1+x_2+x_3\equiv n \md{\Da}$.
Invoking
Lemma~\ref{lem:nonzero}
we see that 
if $\sigma_{\b{a},n}(\Da)>0$ then
\[
\sigma_{\b{a},n}(\Da)
\prod_{i=1}^3\c{L}_{a_i}
\geq
\Da
\prod_{i=1}^3 
\delta_{a_i}(x_i\hspace{-0,2cm}\mod{\Da})
\gg
\Da^{-2}
\prod_{i=1}^3\frac{\phi(h_{a_i})}{h_{a_i}}
.\]
Recalling~\eqref{def:mytonga}
we obtain
$
\Da
\leq
[\Delta_1,\Delta_2,\Delta_3]
\leq
|
\Delta_{1}
\Delta_{2}
\Delta_{3}
|
$,
hence 
\begin{equation}
\label{eq:probus}
\sigma_{\b{a},n}(\Da)
\prod_{i=1}^3\c{L}_{a_i}
\gg
\prod_{i=1}^3\frac{\phi(h_{a_i})}{|
\Delta_{i}
|^2 h_{a_i}}
,\end{equation}
with an absolute implied constant.    
Combined with 
Lemma~\ref{lem:nmiddeducomeon2},
this concludes the 
proof of~\eqref{th:grhreplaceall}.

\subsection{The proof of Theorem \ref{thm:const}}\label{proofthmrthcoro}
The proof of the first part of Theorem \ref{thm:const},
which is~\eqref{ee:grhtyuiop} 
is spread throughout~\S\ref{proofofee:grhtyuiop}.
The proof of the second (and last) 
part of Theorem \ref{thm:const},
which is~\eqref{th:grhreplaceall},
is spread throughout~\S\ref{proofofth:grhreplaceall}.

\subsection{The proof of Corollary~\ref{cor:loc_glob}}\label{proofcoro}
Obviously, \emph{(1)} implies \emph{(2)}. For the reverse direction,
let $d\in\{3,\Da\}$ and let $p_1,p_2,p_3$ be primes not dividing $2d$,
such that each $a_i$ is a primitive root modulo $p_i$ and
$p_1+p_2+p_3\equiv n\bmod{d}$. Thus, for every $i=1,2,3$ the
progression $p_i\md{d}$ satisfies $\gcd(p_i,d)=1$ and contains an odd
prime having $a_i$ as a primitive root.  We can now use the following
observation due to Lenstra~\cite[p.g.216]{MR0480413}: if $\gcd(x,d)=1$
and $\delta_a(x\hspace{-0,2cm}\mod{d})=0$ then either there is no
prime $p\equiv x \md{d}$ with $\mathbb{F}_p^*= \langle a \rangle$ or
there is one such prime, which must be equal to $2$.  This shows that
we must have $\delta_a(x_i\hspace{-0,2cm}\mod{d})>0$ for every
$i=1,2,3$.  Using the fact that $x_1+x_2+x_3\equiv n \md{d}$, as well
as Definition~\eqref{def:sigmapenotdelta} shows that
$\sigma_{\b{a},n}(\Da)\sigma_{\b{a},n}(3)>0$. By Lemma \ref{lem:nmiddeducomeon2}, we get $\consta(n)>0$, and thus $\consta(n)\gg 1$ by \eqref{th:grhreplaceall}. Thus, \emph{(1)} follows immediately from Theorem \ref{thm:main} and the trivial estimate
\[
\sum_{\substack{p_1+p_2+p_3=n\\
\exists i:\ p_i \mid 6\Delta_{1}\Delta_{2}\Delta_{3}
}}
\Bigg(
\prod_{i=1}^3 \log p_i
\Bigg)
\ll
n
(\log n)^3
.\]

\subsection{The proof of
Theorem~\ref{propositionjoe}}
\label{s:simplifying}
First note that
$\Daaa=|\Delta_a|$.
It is clear that 
for the proof of Theorem~\ref{propositionjoe}
we need
to find equivalent conditions for $n$ to satisfy 
$$
\sigma_{(a,a,a),n}(|\Delta_a|)
\prod_{p\nmid \Delta_a}
\sigma_{(a,a,a),n}(p)
>
0.
$$
By
Lemma~\ref{lem:nmiddeducomeon2}
the condition
$\prod_{p\nmid \Delta_a}
\sigma_{(a,a,a),n}(p)
\neq 0$
is equivalent to  
\begin{equation}
\label{eq:withf58}
\begin{cases}  
n\equiv 3\md{6},&\mbox{if } 3\mid h_a \text{ and } 3\nmid \Delta_a,\\  
n\equiv 1\md{2},&\mbox{otherwise. } \end{cases}
\end{equation}
Hence it remains to find
equivalent conditions for $n$
to satisfy 
$\sigma_{(a,a,a),n}(|\Delta_a|)
>0$. 

\begin{proposition}\label{lem:twofoldpart1}
 Assume that $n$ is an odd positive 
integer.
 \begin{enumerate}
 \item If $3\nmid \gcd(\Delta_a,h_a)$ or $3\mid n$, and if $\Delta_a$ has a prime divisor that is greater than $7$, then $\sigma_{(a,a,a),n}(|\Delta_a|)>0$.
 \item If $3\mid\gcd(\Delta_a,h_a)$ and $3\nmid n$, then $\sigma_{(a,a,a),n}(|\Delta_a|)=0$.
 \end{enumerate}
\end{proposition}

\begin{proof}
  It can be seen directly from Definition~\ref{def:deltaa}
that the quantity $\delta_a(x_i\hspace{-0,2cm}\mod{|\Delta_a|})$
is non-zero if and only if
\begin{equation}\label{eq:x_i_conditions}
\gcd(x_i-1,\Delta_a,h_a)=1,\ \gcd(x_i,\Delta_a)=1 \text{ and } \left(\frac{\Delta_a}{x_i}\right)=-1.
\end{equation}
In view of Definition \ref{def:sigmapenotdelta}, we need to find conditions under which there are $x_1,x_2,x_3\in\ZZ$ with $x_1+x_2+x_3\equiv n\bmod{\Delta_a}$, such that each $x_i$ satisfies \eqref{eq:x_i_conditions}.

To prove \emph{(2)}, we observe that the first two conditions in \eqref{eq:x_i_conditions} imply that $x_i\equiv 2\bmod 3$, hence $3\mid n$.

Let us now prove \emph{(1)}. Write $\Delta_a=\prod_{p\mid \Delta_a}D_p$, where $D_2\in\{-8,-4,8\}$ and $D_p=(-1)^{(p-1)/2}p$ for $p\geq 3$. Let $p'>7$ be the largest prime divisor of $\Delta_a$. For every $p<p'$, we find $x_1^{(p)},x_2^{(p)},x_3^{(p)}\bmod {D_p}$ that solve the congruence $x_1^{(p)}+x_2^{(p)}+x_3^{(p)}\equiv n\bmod{D_p}$ and satisfy $\gcd(x_i^{(p)}-1,\Delta_a,h_a)=\gcd(x_i^{(p)},\Delta_a)=1$. If $p>3$, this is possible for every $n$ by a simple application of the Cauchy--Davenport Theorem. If $p=3$, it is possible precisely by our assumption that then $3\nmid h_a$ or $3\mid n$. Finally, for $p=2$, it is possible since $2\nmid nh_a$.

Let us now define $x^{(p')}_{i}$. Consider the sets

\begin{equation*}
R:=
\Big\{x\in \Z/p'\Z
:\l(\frac{x}{p'}\r)=1,x\neq 1\md{p'}\Big\},\ 
N :=\Big\{x\in \Z/p'\Z
:\l(\frac{x}{p'}\r)=-1\Big\}
.
\end{equation*}
If $\prod_{\substack{p\mid\Delta_a\\p<p'}}\Big(\frac{D_p}{x_i^{(p)}}\Big)=1$, we pick $x_i^{(p')}\in N$, and if $\prod_{\substack{p\mid\Delta_a\\p<p'}}\Big(\frac{D_p}{x_i^{(p)}}\Big)=-1$, we pick $x_i^{(p')}\in R$. We can always do so and achieve $x_1^{(p')}+x_2^{(p')}+x_3^{(p')}\equiv n\bmod{p'}$, as the sets
\begin{equation*}
  R+R+R,\ R+R+N,\ R+N+N,\ N+N+N 
\end{equation*}
cover all of $\ZZ/p'\ZZ$. This follows from a direct computation if $p'=11$ and from the Cauchy--Davenport Theorem if $p'\geq 13$.

To finish our proof of \emph{(1)}, we pick integers $x_i$ that satisfy $x_i\equiv x_i^{(p)}\bmod{D_p}$ for all $p\mid \Delta_a$. Then quadratic reciprocity ensures that
\begin{equation*}
\bigg(\frac{\Delta_a}{x_i}\bigg)=
\Bigg(\frac{x_i^{(p')}}{p'}\Bigg)
\prod_{\substack{p\mid\Delta_a\\p<p'}}
\bigg(\frac{D_p}{x_i^{(p)}}\bigg)=-1
\end{equation*}
for all $i$. Hence, the $x_i$ satisfy \eqref{eq:x_i_conditions}, and moreover $x_1+x_2+x_3\equiv n\bmod{\Delta_a}$.
\end{proof}
\begin{proof}[\textbf{Proof of 
Theorem~\ref{propositionjoe}
}\!\!\!\!
] 
First let us note that the fundamental discriminants with every prime smaller than $11$
are of the form
\[
D_2^{i_1}
(-3)^{i_2}
5^{i_3}
(-7)^{i_4}
,\]
where $D_2$ is an integer in the set $\{-4,8,-8\}$
and every exponent $i_j$ is either $0$ or $1$.
This gives a finite set of values for $\Delta_a$ and 
it is straightforward to use a computer program 
that finds all congruence classes $n\md{\Delta_a}$
such that 
$n\equiv x_1+x_2+x_3\md{\Delta_a}$
for some $\b{x}\in (\Z/\Delta_a\Z)^3$
satisfying all of
the conditions \eqref{eq:x_i_conditions} for $1\leq i \leq 3$.

By Definition~\ref{def:deltaa}
these conditions are equivalent to
$\delta_a(x_i\hspace{-0,2cm}\mod{|\Delta_a|})\neq 0$
and when combined with~\eqref{eq:withf58}
they provide
the congruence classes for $n$
in every row of the table in  
Theorem~\ref{propositionjoe}
apart from 
the last two rows.
For the last two rows, $\Delta_a$ has a prime factor greater than $7$, so one sees by
Proposition~\ref{lem:twofoldpart1}  
that we only have to provide 
conditions on $n$
that are 
equivalent 
to
$\prod_{p\nmid \Delta_a}
\sigma_{(a,a,a),n}(p)
>0$,
which was already 
done in~\eqref{eq:withf58}.
\end{proof} 

\subsection{Non-factorisation of $\consta(n)$}
\label{s:nonfa}
We
finish by showing that the right side  in~\eqref{ee:grhtyuiop}
does not always factorise as an Euler product of a specific form.
Namely, assume that for every non-square integer $a\neq -1$ 
we are given a 
sequence of real numbers $\lambda_a:\Z^2\to [0,\infty)$
such that for every prime $p$ and integers $x,x'$ 
we have  
\begin{equation}
\label{eq:selfcont1}
\delta_a(x\hspace{-0,3cm}\mod{p})>0
\Rightarrow
\lambda_a(x,p)>0
\end{equation}
and 
\[
x\equiv x'\md{p}
\Rightarrow
\lambda_a(x,p)=
\lambda_a(x',p)
.\]
Now, in parallel with~\eqref{def:sigmapenotdelta}
let us define
\[
\varpi_{p,a}(n)
:=
\Bigg(
\sum_{\substack{b_1,b_2,b_3 \md{p}\\ b_1+b_2+b_3\equiv n \md{p}}} 
\prod_{i=1}^3 
\lambda_a(x,p)
\Bigg)
\Bigg(\sum_{\substack{b_1,b_2,b_3 \md{p}\\ b_1+b_2+b_3\equiv n \md{p}}} \hspace{-0,2cm}
\frac{1}{p^3}
\Bigg)^{-1}
.\]
The fact that the quantities $\varpi_{p,a}(n)$
are well-defined follows from the periodicity of $\lambda_a$.

We will see that one cannot have 
the following factorisation
for all odd
integers $n$,
\begin{equation}
\label{eq:selfcont2}
\c{L}_a^3
\sigma_{(a,a,a),n}(|\Delta_a|)
=
\prod_{p\mid \Delta_a}
\varpi_{p,a}(n)
.\end{equation}
Indeed, if
 $a:=(-15)^5=-759375$
then 
 by Definition~\ref{def:deltaa}
we easily see that 
\[
\delta_{-759375}(x\hspace{-0,3cm}\mod{15})>0
\Leftrightarrow
x\md{15} \in \{7,13,14\md{15}\}
,\]
hence 
 for all integers $n\equiv 7\md{15}$
we have
$\sigma_{(a,a,a),n}(|\Delta_a|)=0$ due to~\eqref{def:sigmapenotdelta}
and the fact that for all $\b{x} \in \{7,13,14\}^3$
one has 
$\sum_{i=1}^3
x_i
\neq 7 \md{15}$.
Definition~\ref{def:deltaa}
furthermore 
implies that 
\[
\delta_{-759375}(x\hspace{-0,3cm}\mod{3})>0
\Leftrightarrow
x\md{3} \in \{1,2\md{3}\}
\]
and
\[
\delta_{-759375}(y\hspace{-0,3cm}\mod{5})>0
\Leftrightarrow
y\md{5} \in \{2,3,4\md{5}\}
,\]
therefore
whenever
$n\equiv 7\md{15}$
then the
vectors
$\b{x}=(1,1,2)$
and
$\b{y}=(4,4,4)$
satisfy 
\[
\sum_{i=1}^3 x_i \equiv n \md{3},
\sum_{i=1}^3 y_i \equiv n \md{5}
\text{ and }
\
\prod_{i=1}^3
\delta_{-759375}(x_i\hspace{-0,3cm}\mod{3})
\delta_{-759375}(y_i\hspace{-0,3cm}\mod{5})
>0
.\]
By~\eqref{eq:selfcont1}
this implies that 
$\varpi_{3,-759375}(n)>0,
\varpi_{5,-759375}(n)>0$,
which contradicts~\eqref{eq:selfcont2}
due to
$\sigma_{(a,a,a),n}(|\Delta_a|)=0$.

\end{document}